\documentclass{siamltex}%
\usepackage{amsmath}
\usepackage{amsfonts}
\usepackage{sw20siam}
\usepackage{multicol}
\usepackage{amssymb}
\usepackage{graphicx}%
\providecommand{\U}[1]{\protect\rule{.1in}{.1in}}
\newtheorem{algorithm}[theorem]{Algorithm}

\newtheorem{example}[theorem]{Example}

\newtheorem{remark}[theorem]{Remark}

\begin{document}

\title{On the Infinite Swapping Limit for Parallel Tempering}
\author{Paul Dupuis\thanks{Division of Applied Mathematics, Brown University,
Providence, RI 02912. Research supported in part by the Department of Energy
(DE-SCOO02413), the National Science Foundation (DMS-1008331), and the Air
Force Office of Scientific Research (FA9550-07-1-0544, FA9550-09-1-0378).}
\and Yufei Liu\thanks{Division of Applied Mathematics, Brown University,
Providence, RI 02912. Research supported in part by the Department of Energy
(DE-SCOO02413) and the National Science Foundation (DMS-1008331).}
\and Nuria Plattner\thanks{Department of Chemistry, Brown University, Providence,
RI 02912. Research supported in part by the Department of Energy
(DE-SCOO02413) and the by postdoctoral support through the Swiss National
Science Foundation.}
\and J.D. Doll\thanks{Department of Chemistry, Brown University, Providence, RI
02912. Research supported in part by the Department of Energy (DE-SCOO02413
and departmental program DE-00015561).}}
\date{11 October 2005}
\maketitle

\begin{abstract}
Parallel tempering, also known as replica exchange sampling, is an important
method for simulating complex systems. In this algorithm simulations are
conducted in parallel at a series of temperatures, and the key feature of the
algorithm is a swap mechanism that exchanges configurations between the
parallel simulations at a given rate. The mechanism is designed to allow the
low temperature system of interest to escape from deep local energy minima
where it might otherwise be trapped, via those swaps with the higher
temperature components. In this paper we introduce a performance criteria for
such schemes based on large deviation theory, and argue that the rate of
convergence is a monotone increasing function of the swap rate. This motivates
the study of the limit process as the swap rate goes to infinity. We construct
a scheme which is equivalent to this limit in a distributional sense, but
which involves no swapping at all. Instead, the effect of the swapping is
captured by a collection of weights that influence both the dynamics and the
empirical measure. While theoretically optimal, this limit is not
computationally feasible when the number of temperatures is large, and so
variations that are easy to implement and nearly optimal are also developed.

\end{abstract}



\section{Introduction}

The problem of computing integrals with respect to Gibbs measures occurs in
chemistry, physics, statistics, engineering and elsewhere. In many situations,
there are no viable alternatives to methods based on Monte Carlo. Given an
energy potential, there are standard methods to construct a Markov process
whose unique invariant distribution is the associated Gibbs measure, and an
approximation is given by the occupation or empirical measure of the process
over some finite time interval \cite{liu}. However, a weakness of these
methods is that they may be slow to converge. This happens when the dynamics
of the process do not allow all important parts of the state space to
communicate easily with each other. In large scale applications this occurs
frequently, since the potential function often has complex structures
involving multiple deep local minima.

An interesting method called \textquotedblleft parallel
tempering\textquotedblright\ has been designed to overcome some of the
difficulties associated with rare transitions \cite{eardee, gey, sugoka,
swewan}. In this technique, simulations are conducted in parallel at a series
of temperatures. This method does not require detailed knowledge of or
intricate constructions related to the energy surface and is a standard method
for simulating complex systems. To illustrate the main idea, we first discuss
the diffusion case with two temperatures. Discrete time models will be
considered later in the paper, and there are obvious analogues for discrete
state systems.

Suppose that the probability measure of interest is $\mu(dx)\propto
e^{-V(x)/\tau_{1}}dx$, where $\tau_{1}$ is the temperature and $V:\mathbb{R}%
^{d}\rightarrow\mathbb{R}$ is the potential function. The normalization
constant of this distribution is typically unknown. Under suitable conditions
on $V$, $\mu$ is the unique invariant distribution of the solution to the
stochastic differential equation
\[
dX=-\nabla V(X)dt+\sqrt{2\tau_{1}}dW,
\]
where $W$ is a $d$-dimensional standard Wiener process. A straightforward
Monte Carlo approximation to $\mu$ is the empirical measure over a large time
interval of length $T$, namely,
\[
\frac{1}{T}\int_{B}^{T+B}\delta_{X(t)}(dx)dt,
\]
where $\delta_{x}$ is the Dirac measure at $x$ and $B>0$ denotes a
\textquotedblleft burn-in\textquotedblright\ period. When $V$ has multiple
deep local minima and the temperature $\tau_{1}$ is small, the diffusion $X$
can be trapped within these deep local minima for a long time before moving
out to other parts of the state space. This is the main cause for the inefficiency.

Now consider a second, larger temperature $\tau_{2}$. If $W_{1}$ and $W_{2}$
are independent Wiener processes, then of course the empirical measure of the
pair
\begin{align}
dX_{1}  &  =-\nabla V(X_{1})dt+\sqrt{2\tau_{1}}dW_{1} \label{eq:two_temp_diff}%
\\
dX_{2}  &  =-\nabla V(X_{2})dt+\sqrt{2\tau_{2}}dW_{2}\nonumber
\end{align}
gives an approximation to the Gibbs measure with density%
\begin{equation}
\pi(x_{1},x_{2})\propto e^{-\frac{V(x_{1})}{\tau_{1}}}e^{-\frac{V(x_{2})}%
{\tau_{2}}}. \label{eq:two_temp_dens}%
\end{equation}
The idea of parallel tempering is to allow \textquotedblleft
swaps\textquotedblright\ between the components $X_{1}$ and $X_{2}$. In other
words, at random times the $X_{1}$ component is moved to the current location
of the $X_{2}$ component, and vice versa. Swapping is done according to a
state dependent intensity, and so the resulting process is actually a Markov
jump diffusion. The form of the jump intensity can be selected so that the
invariant distribution remains the same, and thus the empirical measure of
$X_{1}$ can still be used to approximate $\mu$. Specifically, the jump
intensity or swapping intensity is of the Metropolis form $ag(X_{1},X_{2})$,
where%
\begin{equation}
g(x_{1},x_{2})=1\wedge\frac{\pi(x_{2},x_{1})}{\pi(x_{1},x_{2})}
\label{eq:swap_rate}%
\end{equation}
and $a\in(0,\infty)$ is a constant. Note that the calculation of $g$ does not
require the knowledge of the normalization constant. A straightforward
calculation shows that $\pi$ is the stationary density for the resulting
process for all values of $a$ [see (\ref{eqn:echeverria})]. We refer to $a$ as
the \textquotedblleft swap rate,\textquotedblright\ and note that as $a$
increases, the swaps become more frequent.

The intuition behind parallel tempering is that the higher temperature
component, being driven by a Wiener process with greater volatility, will move
more easily between the different parts of the state space. This
\textquotedblleft ease-of-movement\textquotedblright\ is transferred to the
lower temperature component via the swapping mechanism so that it is less
likely to be trapped in the deep local minima of the energy potential. This,
in turn, is expected to lead to more rapid convergence of the empirical
measure to the invariant distribution of the low temperature component. There
is an obvious extension to more than two temperatures.

Although this procedure is remarkably simple and needs little detailed
information for implementation,
%(besides $\nabla V$ in this continuous time context),
relatively little is known regarding theoretical properties. A number of
papers discuss the efficiency and optimal design of parallel tempering
\cite{kof, preprecio1, preprecio2}. However, most of these discussions are
based on heuristics and empirical evidence. In general, some care is required
to construct schemes that are effective. For example, it can happen that for a
given energy potential function and swapping rate, the probability for
swapping may be so low that it does not significantly improve performance.

There are two aims to the current paper. The first is to introduce a
performance criteria for Monte Carlo schemes of this general kind that differs
in some interesting ways from traditional criteria, such as the magnitude of
the sub-dominant eigenvalue of a related operator \cite{meytwe, wooschhub}.
More precisely, we use the theory of large deviations to define a
\textquotedblleft rate of convergence\textquotedblright\ for the empirical
measure. The key observation here is that this rate, and hence the performance
of parallel tempering, is monotonically increasing with respect to the swap
rate $a$. Traditional wisdom in the application of parallel tempering has been
that one should not attempt to swap too frequently. While an obvious reason is
that the computational cost for swapping attempts might become a burden, it
was also argued that frequent swapping would result in poor sampling. For a
discussion on prior approaches to the question of how to set the swapping rate
and an argument in favor of frequent swapping, see \cite{sinmenroi,sinemeroi}.

The use of this large deviation criteria and the resulting monotonicity with
respect to $a$ directly suggest the second aim, which is to study parallel
tempering in the limit as $a\rightarrow\infty$. Note that the computational
cost due just to the swapping will increase without bound, even on bounded
time intervals, when $a\rightarrow\infty$. However, we will construct an
alternative scheme, which uses different process dynamics and a weighted
empirical measure. Because this process no longer swaps particle positions, it
and the weighted empirical measure have a well-defined limit as $a\rightarrow
\infty$ which we call infinite swapping. In effect, the swapping is achieved
through the proper choice of weights and state dependent diffusion
coefficients. This is done for the case of both continuous and discrete time
processes with multiple temperatures.

An outline of the paper is as follows. In the next section the swapping model
in continuous time is introduced and the rate of convergence, as measured by a
large deviations rate function, is defined. The alternative scheme which is
equivalent to swapping but which has a well defined limit is introduced, and
its limit as $a\rightarrow\infty$ is identified. The following section
considers the analogous limit model for more than two temperatures, and
discusses certain practical difficulties associated with direct implementation
when the number of temperatures is not small. The continuous time model is
used for illustration because both the large deviation rate and the weak limit
of the appropriately redefined swapping model take a simpler form than those
of discrete time models. However, the discrete time model is what is actually
implemented in practice. To bridge the gap between continuous time diffusion
models and discrete time models, in Section 4 we discuss the idea of infinite
swapping for continuous time Markov jump processes and prove that the key
properties demonstrated for diffusion models hold here as well. We also state
a uniform (in the swapping parameter) large deviation principle. The discrete
time form actually used in numerical implementation is presented in Section 5.
Section 6 returns to the issue of implementation when the number of
temperatures is not small. In particular, we resolve the difficulty of direct
implementation of the infinite swapping models via approximation by what we
call partial infinite swapping models. Section 7 gives numerical examples, and
an appendix gives the proof of the uniform large deviation principle.

\section{Diffusion models with two temperatures}

Although the implementation of parallel tempering uses a discrete time model,
the motivation for the infinite swapping limit is best illustrated in the
setting where the state process is a continuous time diffusion process. It is
in this case that the large deviation rate function, as well as the
construction of a process that is distributionally equivalent to the infinite
swapping limit, is simplest. In order to minimize the notational overhead, we
discuss in detail the two temperature case. The extension to models with
multiple temperatures is obvious and will be stated in the next section.

\subsection{Model setup}

Let $(\bar{X}_{1}^{a},\bar{X}_{2}^{a})$ denote the Markov jump diffusion
process of parallel tempering with swap rate $a$. That is, between swaps (or
jumps), the process follows the diffusion dynamics (\ref{eq:two_temp_diff}).
Jumps occur according to the state dependent intensity function $ag(\bar
{X}_{1}^{a},\bar{X}_{2}^{a})$. At a jump time $t$, the particles swap
locations, that is, $(\bar{X}_{1}^{a}(t),\bar{X}_{2}^{a}(t))=(\bar{X}_{2}%
^{a}(t-),\bar{X}_{1}^{a}(t-))$. Hence for a smooth functions $f:\mathbb{R}%
^{d}\times\mathbb{R}^{d}\rightarrow\mathbb{R}$ the infinitesimal generator of
the process is given by%
\begin{align*}
\mathcal{L}^{a}f(x_{1},x_{2})  &  =-\left\langle \nabla_{x_{1}}f(x_{1}%
,x_{2}),\nabla V(x_{1})\right\rangle -\left\langle \nabla_{x_{2}}f(x_{1}%
,x_{2}),\nabla V(x_{2})\right\rangle \\
&  \quad\quad+\tau_{1}\text{tr}\left[  \nabla_{x_{1}x_{1}}^{2}f(x_{1}%
,x_{2})\right]  +\tau_{2}\text{tr}\left[  \nabla_{x_{2}x_{2}}^{2}f(x_{1}%
,x_{2})\right] \\
&  \quad\quad+ag(x_{1},x_{2})\left[  f(x_{2},x_{1})-f(x_{1},x_{2})\right]  ,
\end{align*}
where $\nabla_{x_{i}}f$ and $\nabla_{x_{i}x_{i}}^{2}f$ denote the gradient and
the Hessian matrix with respect to $x_{i}$, respectively, and tr denotes
trace. Throughout the paper we also assume the growth condition
\begin{equation}
\lim_{r\rightarrow\infty}\inf_{x:|x|\geq r}\left\langle \nabla
V(x),x/|x|\right\rangle =\infty. \label{eq:growth_rate}%
\end{equation}
This condition not only ensures the existence and uniqueness of the invariant
distribution, but also enforces the exponential tightness needed for the large
deviation principle for the empirical measures.

Recall the definition of $\pi$ in (\ref{eq:two_temp_dens}) and let $\mu$ be
the corresponding Gibbs probability distribution, that is,
\[
\mu(dx_{1}dx_{2})=\pi(x_{1},x_{2})dx_{1}dx_{2}\propto e^{-\frac{V(x_{1})}%
{\tau_{1}}}e^{-\frac{V(x_{2})}{\tau_{2}}}dx_{1}dx_{2}.
\]
Straightforward calculations show that for any smooth function $f$ which
vanishes at infinity%
\begin{equation}
\int_{\mathbb{R}^{d}\times\mathbb{R}^{d}}\mathcal{L}^{a}f(x_{1},x_{2}%
)\mu(dx_{1}dx_{2})=0. \label{eqn:echeverria}%
\end{equation}
Since the condition (\ref{eq:growth_rate}) implies that $V(x)\rightarrow
\infty$ as $|x|\rightarrow\infty$, by the Echeverria's Theorem \cite[Theorem
4.9.17]{ethkur}, $\mu$ is the unique invariant probability distribution of the
process $(\bar{X}_{1}^{a},\bar{X}_{2}^{a})$.

\subsection{Rate of convergence by large deviations}

It follows from the previous discussion and the ergodic theorem \cite{bre}
that, for a fixed burn-in time $B$, with probability one%
\[
\lambda_{T}^{a}\doteq\frac{1}{T}\int_{B}^{T+B}\delta_{(\bar{X}_{1}^{a}%
(t),\bar{X}_{2}^{a}(t))}dt\Rightarrow\mu
\]
as $T\rightarrow\infty$. For notational simplicity we assume without loss of
generality that $B=0$ from now on. A basic question of interest is how rapid
is this convergence, and how does it depend on the swap rate $a$? In
particular, what is the rate of convergence of the lower temperature marginal?

We note that standard measures one might use for the rate of convergence,
%for the transition probability
%\[
%p((x_{1},x_{2}),(dy_{1},dy_{2});t)=P\left\{  (\bar{X}_{1}^{a}(t),\bar{X}%
%_{2}^{a}(t))\in dy_{1}\times dy_{2}|(\bar{X}_{1}^{a}(0),\bar{X}_{2}%
%^{a}(0))=(x_{1},x_{2})\right\}  ,
%\]
such as the second eigenvalue of the associated operator, are not necessarily
appropriate here. They only provide indirect information on the convergence
properties of the empirical measure, which is the quantity of interest in the
Monte Carlo approximation. Such measures properly characterize the convergence
rate of the transition probability
\[
p(\boldsymbol{x},d\boldsymbol{y};t)=P\left\{  (\bar{X}_{1}^{a}(t),\bar{X}%
_{2}^{a}(t))\in d\boldsymbol{y}|(\bar{X}_{1}^{a}(0),\bar{X}_{2}^{a}%
(0))=\boldsymbol{x}\right\}  ,~~\boldsymbol{x},\boldsymbol{y}\in\mathbb{R}%
^{d}\times\mathbb{R}^{d},
\]
as $t\rightarrow\infty$. However, they neglect the time averaging effect of
the empirical measure, an effect that is not present with the transition
probability. In fact, it is easy to construct examples such as nearly periodic
Markov chains for which the second eigenvalue suggests a slow convergence when
in fact the empirical measure converges quickly \cite{ros}.

Another commonly used criterion for algorithm performance is the notion of
asymptotic variance \cite{liu, pes, tie}. For a given functional
$f:\mathbb{R}^{d}\times\mathbb{R}^{d}\rightarrow\mathbb{R}$, one can establish
a central limit theorem which asserts that as $T\rightarrow\infty$
\[
\mbox{Var}\left[  \frac{1}{\sqrt{T}}\int_{0}^{T}f(\bar{X}_{1}^{a}(t),\bar
{X}_{2}^{a}(t))\,dt\right]  \rightarrow\sigma^{2}.
\]
The magnitude of $\sigma$ is used to measure the statistical efficiency of the
algorithm. The asymptotic variance is closely related to the spectral
properties of the underlying probability transition kernel \cite{kipvar, ros}.
However, as with the second eigenvalue the usefulness of this criterion for
evaluating performance of the empirical measure $\lambda_{T}^{a}$ is not clear.

In this paper, we use the large deviation rate function to characterize the
rate of convergence of a sequence of random probability measures. To be more
precise, let $S$ be a Polish space, that is, a complete and separable metric
space. Denote by $\mathcal{P}(S)$ the space of all probability measures on $S
$. We equip $\mathcal{P}(S)$ with the topology of weak convergence, though one
can often use the stronger $\tau$-topology \cite{dupell4}. Under the weak
topology, $\mathcal{P}(S)$ is metrizable and itself a Polish space. Note that
the empirical measure $\lambda_{T}^{a}$ is a random probability measure, that
is, a random variable taking values in the space $\mathcal{P}(S)$.

\begin{definition}
A sequence of random probability measures $\left\{  \gamma_{T}\right\}  $ is
said to satisfy a large deviation principle (LDP) with rate function
$I:\mathcal{P}(S)\rightarrow\lbrack0,\infty]$, if for all open sets
$O\subset\mathcal{P}(S)$
\[
\liminf_{T\rightarrow\infty}\frac{1}{T}\log P\left\{  \gamma_{T}\in O\right\}
\geq-\inf_{\nu\in O}I(\nu),
\]
for all closed sets $F\subset\mathcal{P}(S)$%
\[
\limsup_{T\rightarrow\infty}\frac{1}{T}\log P\left\{  \gamma_{T}\in F\right\}
\leq-\inf_{\nu\in F}I(\nu),
\]
and if $\left\{  \nu:I(\nu)\leq M\right\}  $ is compact in $\mathcal{P}(S)$
for all $M<\infty$.
\end{definition}

For our problem all rate functions encountered will vanish only at the unique
invariant distribution $\mu$, and hence give information on the rate of
convergence of $\lambda_{T}^{a}$. A larger rate function will indicate faster
convergence, though this is only an asymptotic statement valid for
sufficiently large $T$.

\subsection{Explicit form of rate function}

The large deviation theory for the empirical measure of a Markov process was
first studied in \cite{donvar1}. Besides the Feller property, which will hold
for all processes we consider, the validity of the LDP depends on two types of
conditions. One is a so-called transitivity condition, which requires that
there are times $T_{1}$ and $T_{2}$ such that for any $\boldsymbol{x}%
,\boldsymbol{y}\in\mathbb{R}^{d}\times\mathbb{R}^{d}$,%
\[
\int_{0}^{T_{1}}e^{-t}p(\boldsymbol{x},d\boldsymbol{z};t)dt\ll\int_{0}^{T_{2}%
}e^{-t}p(\boldsymbol{y},d\boldsymbol{z};t)dt,
\]
where $\ll$ indicates that the measure in $\boldsymbol{z}$ on the left is
absolutely continuous with respect to the measure on the right. For the jump
diffusion process we consider here, this condition holds automatically since
$\nabla V$ is bounded on bounded sets, $g$ is bounded, and the diffusion
coefficients are uniformly non-degenerate. The second type of condition is one
that enforces a strong form of tightness, such as (\ref{eq:growth_rate}).

Under condition (\ref{eq:growth_rate}), the LDP holds for $\left\{
\lambda_{T}^{a}:T>0\right\}  $ and the rate function, denoted by $I^{a}$,
takes a fairly explicit form because the process is in continuous time and
reversible \cite{donvar1, pin}. We will state the following result and omit
the largely straightforward calculation since its role here is motivational.
[A uniform LDP for the analogous jump Markov process will be stated in Section
4, and its proof is given in the appendix.]

Let $\nu$ be a probability measure on $\mathbb{R}^{d}\times\mathbb{R}^{d}$
with smooth density. Define $\theta(x_{1},x_{2})\doteq\lbrack d\nu/d\mu
](x_{1},x_{2})$. Then $I^{a}(\nu)$ can be expressed as%
\begin{equation}
I^{a}(\nu)=J_{0}(\nu)+aJ_{1}(\nu), \label{eq:a_dep_rate_cont_time}%
\end{equation}
where
\begin{align*}
J_{0}(\nu)  &  =\int_{\mathbb{R}^{d}\times\mathbb{R}^{d}}\frac{1}%
{8\theta(x_{1},x_{2})^{2}}\left[  \tau_{1}\left\Vert \nabla_{x_{1}}%
\theta(x_{1},x_{2})\right\Vert ^{2}+\tau_{2}\left\Vert \nabla_{x_{2}}%
\theta(x_{1},x_{2})\right\Vert ^{2}\right]  \nu(dx_{1}dx_{2})\\
J_{1}(\nu)  &  =\int_{\mathbb{R}^{d}\times\mathbb{R}^{d}}g(x_{1},x_{2}%
)\ell\left(  \sqrt{\frac{\theta(x_{2},x_{1})}{\theta(x_{1},x_{2})}}\right)
\nu(dx_{1}dx_{2}),
\end{align*}
and where $\ell\left(  z\right)  =z\log z-z+1$ for $z\geq0$ is familiar from
the large deviation theory for jump processes.
%Moreover one can show that $\nu$ is
%the invariant distribution of a process of the same form as the original jump
%diffusion, but with drift
%\[
%\left(  \frac{\varepsilon_{1}}{2\theta(x_{1},x_{2})}\frac{\partial}{\partial
%x_{1}}\theta(x_{1},x_{2})-\frac{\partial}{\partial x_{1}}V(x_{1}%
%),\frac{\varepsilon_{2}}{2\theta(x_{1},x_{2})}\frac{\partial}{\partial x_{2}%
%}\theta(x_{1},x_{2})-\frac{\partial}{\partial x_{2}}V(x_{2})\right)
%\]
%and jump intensity
%\[
%a\bar{g}(x_{1},x_{2})=ag(x_{1},x_{2})\sqrt{\frac{\theta(x_{2},x_{1})}%
%{\theta(x_{1},x_{2})}}.
%\]

The key observation is that the rate function $I^{a}(\nu)$ is affine in the
swapping rate $a$, with $J_{0}(\nu)$ the rate function in the case of no
swapping. Furthermore, $J_{1}(\nu)\geq0$ with equality if and only if
$\theta(x_{2},x_{1})=\theta(x_{1},x_{2})$ for $\nu$-a.e.\negthinspace
\ $(x_{1},x_{2})$. This form of the rate function, and in particular its
monotonicity in $a$, motivates the study of the \emph{infinite swapping limit}
as $a\rightarrow\infty$.

\begin{remark}
\label{remark:limit-rate} \emph{The limit of the rate function }$I^{a}%
$\emph{\ satisfies}%
\[
I^{\infty}(\nu)\doteq\lim_{a\rightarrow\infty}I^{a}(\nu)=\left\{
\begin{array}
[c]{cl}%
J_{0}(\nu) & \theta(x_{1},x_{2})=\theta(x_{2},x_{1})\text{ }\nu\text{-a.s.,}\\
\infty & \text{otherwise.}%
\end{array}
\right.
\]
\emph{Hence for }$I^{\infty}(\nu)$\emph{\ to be finite it is necessary that
}$\nu$\emph{\ put exactly the same relative weight as }$\mu$ \emph{on the
points }$(x_{1},x_{2})$\emph{\ and }$(x_{2},x_{1})$\emph{. Note that if a
process could be constructed with }$I^{\infty}$ \emph{as its rate function,
then with the large deviation rate as our criteria such a process improves on
parallel tempering with finite swapping rate in exactly those situations where
parallel tempering improves on the process with no swapping at all. }
\end{remark}

\subsection{Infinite swapping limit}

From a practical perspective, it may appear that there are limitations on how
much benefit one obtains by letting $a\rightarrow\infty$. When implemented in
discrete time, the overall jump intensity corresponds to the generation of
roughly $a$ independent random variables that are uniform on $[0,1]$ for each
corresponding unit of continuous time, and based on each uniform variable a
comparison is made to decide whether or not to make the swap. Hence even for
fixed and finite $T$, the computations required to simulate a single
trajectory scale like $a$ as $a\rightarrow\infty$. Thus it is of interest if
one can gain the benefit of the higher swapping rate without all the
computational burden. \ This turns out to be possible, but requires that we
view the prelimit processes in a different way.

It is clear that the processes $(\bar{X}_{1}^{a},\bar{X}_{2}^{a})$ are not
tight as $a\rightarrow\infty$, since the number of discontinuities of size
$O(1)$ will grow without bound in any time interval of positive length. In
order to obtain a limit, we consider alternative processes defined by%
\begin{align}
d\bar{Y}_{1}^{a}  &  =-\nabla V(\bar{Y}_{1}^{a})dt+\sqrt{2\tau_{1}1_{\{\bar
{Z}^{a}=0\}}+2\tau_{2}1_{\{\bar{Z}^{a}=1\}}}\,dW_{1}\label{eq:prelimit_swap}\\
d\bar{Y}_{2}^{a}  &  =-\nabla V(\bar{Y}_{2}^{a})dt+\sqrt{2\tau_{2}1_{\{\bar
{Z}^{a}=0\}}+2\tau_{1}1_{\{\bar{Z}^{a}=1\}}}\,dW_{2}\nonumber
\end{align}
where $\bar{Z}^{a}$ is a jump process that switches from state $0$ to state
$1$ with intensity $ag(\bar{Y}_{1}^{a},\bar{Y}_{2}^{a})$ and from state $1$ to
state $0$ with intensity $ag(\bar{Y}_{2}^{a},\bar{Y}_{1}^{a})$.
%As before, this process can be
%constructed in terms of a sequence of exponential random variables.
Compared to conventional parallel tempering, the processes $(\bar{Y}_{1}%
^{a},\bar{Y}_{2}^{a})$ swap the diffusion coefficients at the jump times
rather than the physical locations of two particles with constant diffusion
coefficients. For this reason, we refer to the solution to
(\ref{eq:prelimit_swap}) as the \textit{temperature swapped process, }in order
to distinguish it from the \textit{particle swapped process} $(\bar{X}_{1}%
^{a},\bar{X}_{2}^{a})$. We illustrate these processes in Figure
\ref{fig:2_temp}. Note that the solid line and the dotted line represent
$\bar{X}_{1}^{a}$ and $\bar{X}_{2}^{a}$, respectively. These processes have
more and more frequent jumps of size $O(1)$ as $a\rightarrow\infty$. In
contrast, the process $(\bar{Y}_{1}^{a},\bar{Y}_{2}^{a})$ have varying
diffusion coefficient. The figure attempts to also suggest features of the
discrete time setting, with both successful and failed swap attempts.%

%TCIMACRO{\FRAME{ftbpFU}{4.1295in}{1.3932in}{0pt}{\Qcb{Temperature swapped and
%particle swapped processes}}{\Qlb{fig:2_temp}}{swapfig1.eps}%
%width 4.1295in;  height 1.3932in;  depth 0pt;  original-width 10.2532in;
%original-height 3.4134in;  cropleft "0";  croptop "1";  cropright "1";
%cropbottom "0";  filename 'swapfig1.eps';file-properties "XNPEU";}}}%
%BeginExpansion
\begin{figure}[ptb]%
\centering
\includegraphics[
height=1.3932in,
width=4.1295in
]%
{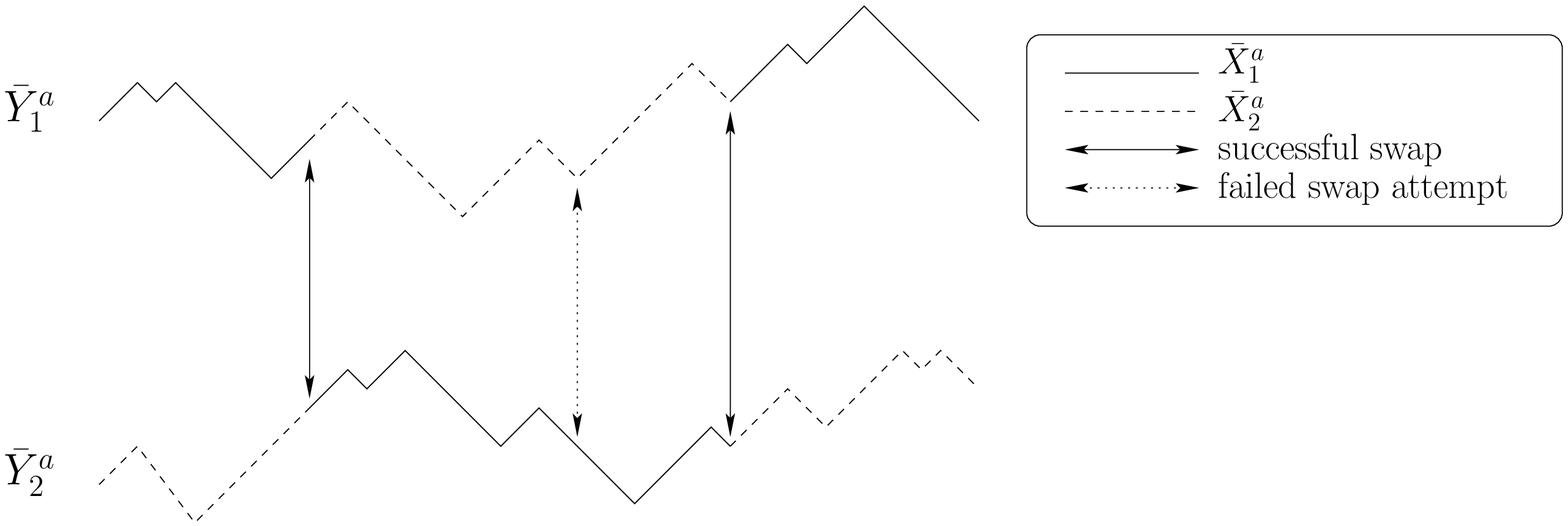}%
\caption{Temperature swapped and particle swapped processes}%
\label{fig:2_temp}%
\end{figure}
%EndExpansion

Clearly the empirical measure of $(\bar{Y}_{1}^{a},\bar{Y}_{2}^{a})$ does not
provide an approximation to $\mu$. Instead, we should shift attention between
$(\bar{Y}_{1}^{a},\bar{Y}_{2}^{a})$ and $(\bar{Y}_{2}^{a},\bar{Y}_{1}^{a})$
depending on the value of $\bar{Z}^{a}$. Indeed, the random probability
measures%
\begin{equation}
\eta_{T}^{a}=\frac{1}{T}\int_{0}^{T}\left[  1_{\left\{  \bar{Z}^{a}%
(t)=0\right\}  }\delta_{(\bar{Y}_{1}^{a}(t),\bar{Y}_{2}^{a}(t))}+1_{\left\{
\bar{Z}^{a}(t)=1\right\}  }\delta_{(\bar{Y}_{2}^{a}(t),\bar{Y}_{1}^{a}%
(t))}\right]  dt \label{eq:pre_emp}%
\end{equation}
have the same distribution as
\[
\frac{1}{T}\int_{0}^{T}\delta_{(\bar{X}_{1}^{a}(s),\bar{X}_{2}^{a}(s))}ds,
\]
and hence converge to $\mu$ at the same rate. However, these processes and
measures have well defined limits in distribution as $a\rightarrow\infty$.
More precisely, we have the following result. For the proof see \cite{kus84}.
Related (but more complex) calculations are needed to prove the uniform large
deviation result given in Theorem \ref{thm:LDP_MP}.

\begin{theorem}
\label{thm:inf_swp_lim}Assume that $\nabla V$ is locally Lipschitz continuous.
Then for each $T$ the sequence $(\bar{Y}_{1}^{a},\bar{Y}_{2}^{a},\eta_{T}%
^{a})$ converges in distribution to $(\bar{Y}_{1}^{\infty},\bar{Y}_{2}%
^{\infty},\eta_{T})$ as $a\rightarrow\infty$, where $(\bar{Y}_{1}^{\infty
},\bar{Y}_{2}^{\infty})$ is the unique strong solution to
\begin{align}
d\bar{Y}_{1}^{\infty}  &  =-\nabla V(\bar{Y}_{1}^{\infty})dt+\sqrt{2\tau
_{1}\rho(\bar{Y}_{1}^{\infty},\bar{Y}_{2}^{\infty})+2\tau_{2}\rho(\bar{Y}%
_{2}^{\infty},\bar{Y}_{1}^{\infty})}dW_{1}\label{eq:limit_swap}\\
d\bar{Y}_{2}^{\infty}  &  =-\nabla V(\bar{Y}_{2}^{\infty})dt+\sqrt{2\tau
_{2}\rho(\bar{Y}_{1}^{\infty},\bar{Y}_{2}^{\infty})+2\tau_{1}\rho(\bar{Y}%
_{2}^{\infty},\bar{Y}_{1}^{\infty})}dW_{2},\nonumber
\end{align}%
\begin{equation}
\eta_{T}^{\infty}=\frac{1}{T}\int_{0}^{T}\left[  \rho(\bar{Y}_{1}^{\infty
}(t),\bar{Y}_{2}^{\infty}(t))\delta_{(\bar{Y}_{1}^{\infty}(t),\bar{Y}%
_{2}^{\infty}(t))}+\rho(\bar{Y}_{2}^{\infty}(t),\bar{Y}_{1}^{\infty}%
(t))\delta_{(\bar{Y}_{2}^{\infty}(t),\bar{Y}_{1}^{\infty}(t))}\right]  dt,
\label{eq:weighted_emp}%
\end{equation}
and%
\[
\rho(x_{1},x_{2})\doteq\frac{\pi(x_{1},x_{2})}{\pi(x_{2},x_{1})+\pi
(x_{1},x_{2})}.
\]

\end{theorem}

The existence and form of the limit are due to the time scale separation
between the fast $\bar{Z}^{a}$ process and the slow $(\bar{Y}_{1}^{a},\bar
{Y}_{2}^{a})$ process. To give an intuitive explanation of the limit dynamics,
consider the prelimit processes (\ref{eq:prelimit_swap}). Suppose that on a
small time interval, the value of the slow process $(\bar{Y}_{1}^{a},\bar
{Y}_{2}^{a})$ does not vary much, say $(\bar{Y}_{1}^{a},\bar{Y}_{2}%
^{a})\approx(x_{1},x_{2})$. Given the dynamics of the binary process $\bar
{Z}^{a}$, it is easy to verify that as $a$ tends to infinity the fractions of
time that $\bar{Z}^{a}=0$ and $\bar{Z}^{a}=1$ are $\rho(x_{1},x_{2})$ and
$\rho(x_{2},x_{1})$, respectively. This leads to the limit dynamics
(\ref{eq:limit_swap}). When mapped back to the particle swapped process,
$\rho(x_{1},x_{2})$ and $\rho(x_{2},x_{1})$ account for the fraction of time
that $(\bar{X}_{1}^{a},\bar{X}_{2}^{a})=(x_{1},x_{2})$ and $(\bar{X}_{1}%
^{a},\bar{X}_{2}^{a})=(x_{2},x_{1})$, respectively, which naturally leads to
the limit weighted empirical measure (\ref{eq:weighted_emp}).

The weights $\rho_{1}$ and $\rho_{2}$ do not depend on the unknown
normalization constant, and in fact
\begin{equation}
\rho(x_{1},x_{2})=\frac{e^{-\frac{V(x_{1})}{\tau_{1}}-\frac{V(x_{2})}{\tau
_{2}}}}{e^{-\frac{V(x_{1})}{\tau_{1}}-\frac{V(x_{2})}{\tau_{2}}}%
+e^{-\frac{V(x_{2})}{\tau_{1}}-\frac{V(x_{1})}{\tau_{2}}}} \label{eqnrho}%
\end{equation}
and
\[
\rho(x_{2},x_{1})=1-\rho(x_{1},x_{2})=\frac{e^{-\frac{V(x_{2})}{\tau_{1}%
}-\frac{V(x_{1})}{\tau_{2}}}}{e^{-\frac{V(x_{1})}{\tau_{1}}-\frac{V(x_{2}%
)}{\tau_{2}}}+e^{-\frac{V(x_{2})}{\tau_{1}}-\frac{V(x_{1})}{\tau_{2}}}}.
\]
The following properties of the limit system are worth noting.\medskip

\begin{enumerate}
\item \textsc{Instantaneous equilibration of multiple locations.} Observe that
the lower temperature component of this modified \textquotedblleft empirical
measure,\textquotedblright\ i.e., the first marginal, uses contributions from
both components at all times, corrected according to the weights. The form of
the weights in (\ref{eq:weighted_emp}) guarantees that the contributions to
$\eta_{T}^{\infty}$ from locations $(\bar{Y}_{1}^{\infty},\bar{Y}_{2}^{\infty
})$ and $(\bar{Y}_{2}^{\infty},\bar{Y}_{1}^{\infty})$ are at any time
perfectly balanced according to the invariant distribution on product space.

\item \textsc{Symmetry and invariant distribution. }While the marginals of
$\eta_{T}^{\infty}$ play very different roles, the dynamics of $\bar{Y}%
_{1}^{\infty}$ and $\bar{Y}_{2}^{\infty}$ are actually symmetric. Using the
Echeverria's Theorem \cite[Theorem 4.9.17]{ethkur}, it can be shown that the
unique invariant distribution of the process $(\bar{Y}_{1}^{\infty},\bar
{Y}_{2}^{\infty})$ has the density
\[
\frac{1}{2}[\pi(x_{1},x_{2})+\pi(x_{2},x_{1})].
\]
It then follows from the ergodic theorem that $\eta_{T}^{\infty}\Rightarrow
\mu$ w.p.1 as $T\rightarrow\infty$. This is hardly surprising, since $\mu$ is
the invariant distribution for the prelimit processes $(\bar{X}_{1}^{a}%
,\bar{X}_{2}^{a})$.

\item \textsc{Escape from local minima. }Finally it is worth commenting on the
behavior of the diffusion coefficients as a function of the relative positions
of $\bar{Y}_{1}^{\infty}$ and $\bar{Y}_{2}^{\infty}$ on an energy landscape.
Recall that $\tau_{1}<\tau_{2}$. Suppose that $\bar{Y}_{1}^{\infty}(t)$ is
near the bottom of a local minimum (which for simplicity we set to be zero),
while $\bar{Y}_{2}^{\infty}(t)$ is at a higher energy level, perhaps within
the same local minimum. Then
\[
\rho(y_{1},y_{2})\approx\frac{e^{-\frac{V(y_{2})}{\tau_{2}}}}{e^{-\frac
{V(y_{2})}{\tau_{2}}}+e^{-\frac{V(y_{2})}{\tau_{1}}}}\approx1,~~~~\rho
(y_{2},y_{1})=1-\rho(y_{1},y_{2})\approx0.
\]
Thus to some degree the dynamics look like%
\begin{align*}
d\bar{Y}_{1}^{\infty}  &  =-\nabla V(\bar{Y}_{1}^{\infty})dt+\sqrt{2\tau_{1}%
}dW_{1}\\
d\bar{Y}_{2}^{\infty}  &  =-\nabla V(\bar{Y}_{2}^{\infty})dt+\sqrt{2\tau_{2}%
}dW_{2},
\end{align*}
i.e., the particle higher up on the energy landscape is given the greater
diffusion coefficient, while the one near the bottom of the well is
automatically given the lower coefficient. Hence the particle which is already
closer to escaping from the well is automatically given the greater noise
(within the confines of $(\tau_{1},\tau_{2})$). Recalling the role of the
higher temperature particle is to more assiduously explore the landscape in
parallel tempering, this is an interesting property.
\end{enumerate}

One can apply results from \cite{donvar1} to show that the empirical measure
of the infinite swapping limit $\{\eta_{T}^{\infty}:T>0\}$ satisfies a large
deviation principle with rate function $I^{\infty}$ as defined in Remark
\ref{remark:limit-rate}. However, to justify the claim that the infinite
swapping model is truly superior to the finite swapping variant (note that
$I^{\infty}\geq I^{a}$ for any finite $a$), one should establish a
\textit{uniform} large deviation principle, which would show that $I^{\infty}$
is the correct rate function for any sequence $\{a_{T}:T>0\}\subset
\lbrack0,\infty]$ such that $a_{T}\rightarrow\infty$ as $T\rightarrow\infty$.
We omit the proof here, since in Theorem \ref{thm:LDP_MP} the analogous result
will be proved in the setting of continuous time jump Markov processes.

\section{Diffusion models with multiple temperatures}

\label{Subsect:K_temp_cont_time}

In practice parallel tempering uses swaps between more than two temperatures.
A key reason is that if the gap between the temperatures is too large then the
probability of a successful swap under the [discrete time version of the]
Metropolis rule (\ref{eq:swap_rate}) is far too low for the exchange of
information to be effective. A natural generalization is to introduce, to the
degree that computational feasibility is maintained, a ladder of higher
temperatures, and then attempt pairwise swaps between particles. There are a
variety of schemes used to select which pair to attempt the swap, including
deterministic and randomized rules for selecting only adjacent temperatures or
arbitrary pair of temperatures. However, if one were to replace any of these
particle swapped processes with its equivalent temperature swapped analogue
and consider the infinite swapping limit, one would get the same system of
process dynamics and weighted empirical measures which we now describe.

Suppose that besides the lowest temperature $\tau_{1}$ (in many cases the
temperature of principal interest), we introduce the collection of higher
temperatures
\[
\tau_{1}<\tau_{2}<\cdots<\tau_{K}.
\]
Let $\boldsymbol{y}=\left(  y_{1},y_{2},\ldots,y_{K}\right)  \in
(\mathbb{R}^{d})^{K}$ be a generic point in the state space of the process and
define a product Gibbs distribution with the density%
\[
\pi\left(  \boldsymbol{y}\right)  =\pi\left(  y_{1},y_{2},\ldots,y_{K}\right)
\propto\left.  e^{-V\left(  y_{1}\right)  /\tau_{1}}e^{-V\left(  y_{2}\right)
/\tau_{2}}\cdots e^{-V\left(  y_{K}\right)  /\tau_{K}}\right.  .
\]
The limit of the temperature swapped processes with $K$ temperatures takes the
form
\begin{align*}
d\bar{Y}_{1}^{\infty}  &  =-\nabla V\left(  \bar{Y}_{1}^{\infty}\right)
dt+\sqrt{2\rho_{11}\tau_{1}+2\rho_{12}\tau_{2}+\ldots+2\rho_{1K}\tau_{K}%
}dW_{1}\\
d\bar{Y}_{2}^{\infty}  &  =-\nabla V\left(  \bar{Y}_{2}^{\infty}\right)
dt+\sqrt{2\rho_{21}\tau_{1}+2\rho_{22}\tau_{2}+\ldots+2\rho_{2K}\tau_{K}%
}dW_{2}\\
&  \vdots\\
d\bar{Y}_{K}^{\infty}  &  =-\nabla V\left(  \bar{Y}_{K}^{\infty}\right)
dt+\sqrt{2\rho_{K1}\tau_{1}+2\rho_{K2}\tau_{2}+\ldots+2\rho_{KK}\tau_{K}%
}dW_{K}.
\end{align*}
To define these weights $\rho_{ij}$ it is convenient to introduce some new notation.

Let $S_{K}$ be the collection of all bijective mappings from $\left\{
1,2,\ldots,K\right\}  $ to itself. $S_{K}$ has $K!$ elements, each of which
corresponds to a unique permutation of the set $\{1,2,\ldots,K\}$, and $S_{K}
$ forms a group with the group action defined by composition. Let $\sigma
^{-1}$ denote the inverse of $\sigma$. Furthermore, for each $\sigma\in S_{K}$
and every $\boldsymbol{y}=\left(  y_{1},y_{2},\ldots,y_{K}\right)
\in(\mathbb{R}^{d})^{K}$, define $\boldsymbol{y}_{\sigma}\doteq\left(
y_{\sigma(1)},y_{\sigma(2)},\ldots,y_{\sigma(K)}\right)  $.

At the level of the prelimit particle swapped process, we interpret the
permutation $\sigma$ to correspond to event that the particles at location
$\boldsymbol{y}=(y_{1},y_{2},\ldots,y_{K})$ are swapped to the new location
$\boldsymbol{y}_{\sigma}=\left(  y_{\sigma(1)},y_{\sigma(2)},\ldots
,y_{\sigma(K)}\right)  $. Under the temperature swapped process, this
corresponds to the event that particles initially assigned temperatures in the
order $\tau_{1},\tau_{2},\ldots,\tau_{K}$ have now been assigned the
temperatures $\tau_{\sigma^{-1}(1)},\tau_{\sigma^{-1}(2)},\ldots,\tau
_{\sigma^{-1}(K)}$.

The identification of the infinite swapping limit of the temperature swapped
processes is very similar to that of the two temperature model in the previous
section. By exploiting the time-scale separation, one can assume that in a
small time interval the only motion is due to temperature swapping and the
motion due to diffusion is negligible. Hence the fraction of time that the
permutation $\sigma$ is in effect should again be proportional to the relative
weight assigned by the invariant distribution to $\boldsymbol{y}_{\sigma}$,
that is,
\[
\pi(\boldsymbol{y}_{\sigma})=\pi\left(  y_{\sigma(1)},y_{\sigma(2)}%
,\ldots,y_{\sigma(K)}\right)  .
\]
Thus if
\[
w(\boldsymbol{y})\doteq\frac{\pi(\boldsymbol{y})}{\sum_{\theta\in S_{K}}%
\pi(\boldsymbol{y}_{\theta})},
\]
then the fraction of time that the permutation $\sigma$ is in effect is
$w(\boldsymbol{y}_{\sigma})$. Note that for any $\boldsymbol{y}$,
\[
\sum_{\sigma\in S_{K}}w(\boldsymbol{y}_{\sigma})=1.
\]
Going back to the definition of the weights $\rho_{ij}(\boldsymbol{y})$,
$i,j=1,\ldots,K$, it is clear that they represent the limit proportion of time
that the $i$-th particle is assigned the temperature $j$ and hence will
satisfy
\[
\rho_{ij}\left(  \boldsymbol{y}\right)  =\sum_{\sigma:~\sigma\left(  j\right)
=i}w\left(  \boldsymbol{y}_{\sigma}\right)  .
\]
Likewise the replacement for the empirical measure, accounting as it does for
mapping the temperature swapped process back to the particle swapped process,
is given by%
\begin{equation}
\eta_{T}^{\infty}=\frac{1}{T}\int_{0}^{T}\sum_{\sigma\in S_{K}}%
w(\boldsymbol{\bar{Y}}_{\sigma}^{\infty}(t))\delta_{\boldsymbol{\bar{Y}%
}_{\sigma}^{\infty}(t)}dt, \label{eq:inf_swap_emp_meas}%
\end{equation}
where $\boldsymbol{\bar{Y}}_{\sigma}^{\infty}(t)\doteq\lbrack\boldsymbol{\bar
{Y}}^{\infty}(t)]_{\sigma}=(\bar{Y}_{\sigma(1)}^{\infty}(t),\bar{Y}%
_{\sigma(2)}^{\infty}(t),\ldots,\bar{Y}_{\sigma(K)}^{\infty}(t))$.

The instantaneous equilibration property still holds for the infinite swapping
system with multiple temperatures. That is, at any time $t\in\lbrack0,T]$ and
given a current position $\boldsymbol{\bar{Y}}^{\infty}(t)=\boldsymbol{y}$,
the weighted empirical measure $\eta_{T}^{\infty}$ has contributions from all
locations of the form $\boldsymbol{y}_{\sigma},\sigma\in S_{K}$, balanced
exactly according to their relative contributions from the invariant density
$\pi\left(  \boldsymbol{y}_{\sigma}\right)  $. The dynamics of
$\boldsymbol{\bar{Y}}^{\infty}$ are again symmetric, and the density of the
invariant distribution at point $\boldsymbol{y}$ is
\[
\frac{1}{K!}\sum_{\sigma\in S_{K}}\pi(\boldsymbol{y}_{\sigma}).
\]

\begin{remark}
\label{remark:Partial} \emph{The infinite swapping process described above
allows the most effective communication between all temperatures, and is the
\textquotedblleft best\textquotedblright\ in the sense that it leads to the
largest large deviation rate function and hence the fastest rate of
convergence. However, computation of the coefficients becomes very demanding
for even moderate values of }$K$\emph{, since one needs to evaluate }%
$K!$\emph{\ terms from all possible permutations. In Section
\ref{sec:discrete} we discuss a more tractable and easily implementable family
of schemes which are essentially approximations to the infinite swapping model
presented in the current section and have very similar performance. We call
the current model the full infinite swapping model since it uses the whole
permutation group }$S_{K}$\emph{, as opposed to the partial infinite swapping
model in Section \ref{sec:discrete} where only subgroups of }$S_{K}%
$\emph{\ are used. }
\end{remark}

\section{Infinite swapping for jump Markov processes}

\label{sec:MP}

The continuous time diffusion model is a convenient vehicle to convey the main
idea of infinite swapping. In practice, however, algorithms are implemented in
discrete time. In this section we discuss continuous time pure jump Markov
processes and the associated infinite swapping limit. The purpose of this
intermediate step is to serve as a bridge between the diffusion and discrete
time Markov chain models. These two types of processes have some subtle
differences regarding the infinite swapping limit which is best illustrated
through the continuous time jump Markov model.

In this section we discuss the two-temperature model, and omit the completely
analogous multiple-temperature counterpart. We will not refer to temperatures
$\tau_{1}$ and $\tau_{2}$ to distinguish dynamics. Instead, let $\alpha
_{1}(x,dy)$ and $\alpha_{2}(x,dy)$ be two probability transition kernels on
$\mathbb{R}^{d}$ given $\mathbb{R}^{d}$. One can think of $\alpha_{i}$ as the
dynamics under temperature $\tau_{i}$ for $i=1,2$. We assume that for each
$i=1,2$ the stationary distribution $\mu_{i}$ associated with the transition
kernel $\alpha_{i}$ admits the density $\pi_{i}$ in order to be consistent
with the diffusion models, and define
\[
\mu=\mu_{1}\times\mu_{2},~~~~\pi(x_{1},x_{2})=\pi_{1}(x_{1})\pi_{2}(x_{2}).
\]
We assume that the kernels are Feller and have a density that is uniformly
bounded with respect to Lebesgue measure. These conditions would always be
satisfied in practice. Finally, we assume that the detailed balance or
reversibility condition holds, that is,
\begin{equation}
\alpha_{i}(x,dz)\pi_{i}(x)dx=\alpha_{i}(z,dx)\pi_{i}(z)dz.
\label{eqn:detailed_balance}%
\end{equation}

\subsection{Model setup}

\label{section:noswap_MP}

In the absence of swapping [i.e., swapping rate $a=0$], the dynamics of the
system are as follows. Let $\boldsymbol{X}^{0}=\{\boldsymbol{X}^{0}%
(t)=(X_{1}^{0}(t),X_{2}^{0}(t)):t\geq0\}$ denote a continuous time process
taking values in $\mathbb{R}^{d}\times\mathbb{R}^{d}$. The probability
transition kernel associated with the embedded Markov chain, denoted by
$\boldsymbol{\bar{X}}^{0}=\{(\bar{X}_{1}^{0}(j),\bar{X}_{2}^{0}%
(j)):j=0,1,\ldots\}$, is
\[
P\{\boldsymbol{\bar{X}}^{0}(j+1)\in(dy_{1},dy_{2})|\boldsymbol{\bar{X}}%
^{0}(j)=(x_{1},x_{2})\}=\alpha_{1}(x_{1},dy_{1})\alpha_{2}(x_{2},dy_{2}).
\]
Without loss of generality, we assume that the jump times occur according to a
Poisson process with rate one. In other words, let $\{\tau_{i}\}$ be a
sequence of independent and identically distributed (iid) exponential random
variables with rate one that are independent of $\boldsymbol{\bar{X}}^{0}$.
Then
\[
\boldsymbol{X}^{0}(t)=\boldsymbol{\bar{X}}^{0}(j),~~\mbox{for~}\sum_{i=1}%
^{j}\tau_{i}\leq t<\sum_{i=1}^{j+1}\tau_{i}.
\]
The infinitesimal generator of $\boldsymbol{X}^{0}$ is such that for a given
smooth function $f$,
\[
\mathcal{L}^{0}f(x_{1},x_{2})=\int_{\mathbb{R}^{d}\times\mathbb{R}^{d}%
}[f(y_{1},y_{2})-f(x_{1},x_{2})]\alpha_{1}(x_{1},dy_{1})\alpha_{2}%
(x_{2},dy_{2}).
\]

Owing to the detailed balance condition (\ref{eqn:detailed_balance}), the
operator $\mathcal{L}^{0}$ is self-adjoint. Using arguments similar to but
simpler than those used to prove the uniform LDP in Theorem \ref{thm:LDP_MP},
the large deviations rate function $I^{0}$ associated with the occupation
measure
\[
\eta_{T}^{0}=\frac{1}{T}\int_{0}^{T}\delta_{\boldsymbol{X}^{0}(t)}\,dt
\]
can be explicitly identified: for any probability measure $\nu$ on
$\mathbb{R}^{d}\times\mathbb{R}^{d}$ with $\nu\ll\mu$ and $\theta=d\nu/d\mu$,%
\[
I^{0}(\nu)=1-\int_{(\mathbb{R}^{d}\times\mathbb{R}^{d})^{2}}\sqrt{\theta
(x_{1},x_{2})\theta(y_{1},y_{2})}\pi(x_{1},x_{2})\alpha_{1}(x_{1}%
,dy_{1})\alpha_{2}(x_{2},dy_{2})\,dx_{1}dx_{2},
\]
and $I^{0}$ is extended to all of $\mathcal{P}(\mathbb{R}^{d}\times
\mathbb{R}^{d})$ by lower semicontinuous regularization.

\subsection{Finite swapping model}

\label{sec:finite_swap}

Denote by $\boldsymbol{X}^{a}=\{(X_{1}^{a}(t),X_{2}^{a}(t)):t\geq0\}$ the
state process of the finite swapping model with swapping rate $a$, and let
$\boldsymbol{\bar{X}}^{a}=\{(\bar{X}_{1}^{a}(j),\bar{X}_{2}^{a}%
(j)):j=0,1,\ldots\}$ be the embedded Markov chain. The probability transition
kernel for $\boldsymbol{\bar{X}}^{a}$ is
\begin{align*}
\lefteqn{P\{\boldsymbol{\bar X}^{a}(j+1)\in(dy_{1},dy_{2})|\boldsymbol{\bar
X}^{a}(j)=(x_{1},x_{2})\}~=~\frac{1}{a+1}\alpha_{1}(x_{1},dy_{1})\alpha
_{2}(x_{2},dy_{2})}\\
&  ~~~~+\frac{a}{a+1}\left[  g(x_{1},x_{2})\delta_{(x_{2},x_{1})}%
(dy_{1},dy_{2})+(1-g(x_{1},x_{2}))\delta_{(x_{1},x_{2})}(dy_{1},dy_{2}%
)\right]  ,
\end{align*}
where $g$ is defined as in (\ref{eq:swap_rate}). Furthermore, let $\{\tau
_{i}^{a}\}$ be a sequence of iid exponential random variables with rate
$(a+1)$ and define
\[
\boldsymbol{X}^{a}(t)=\boldsymbol{\bar{X}}^{a}(j),~~\mbox{for~}\sum_{i=1}%
^{j}\tau_{i}^{a}\leq t<\sum_{i=1}^{j+1}\tau_{i}^{a}.
\]
In other words, the jumps occur according to a Poisson process with rate $a+1
$. Note that there are two types of jumps. At any jump time, with probability
$1/(a+1)$ it will be a jump according to the underlying probability transition
kernels $\alpha_{1}$ and $\alpha_{2}$. With probability $a/(a+1)$ it will be
an attempted swap which will succeed with the probability determined by $g$.
As $a$ grows, the swap attempts become more and more frequent. However, the
time between two consecutive jumps of the first type will have the same
distribution as
\[
S^{a}=\sum_{i=1}^{N^{a}}\tau_{i}^{a}
\]
where $N^{a}$ is a geometric random variable with parameter $1/(a+1)$. It is
easy to argue that for any $a$ the distribution of $S^{a}$ is exponential with
rate one. This observation will be useful when we derive the infinite swapping limit.

The infinitesimal generator $\mathcal{L}^{a}$ of $\boldsymbol{\bar{X}}^{a}$ is
such that for any smooth function $f$ on $\mathbb{R}^{d}\times\mathbb{R}^{d}$
\begin{align*}
\mathcal{L}^{a}f(x_{1},x_{2})  &  =\int_{\mathbb{R}^{d}\times\mathbb{R}^{d}%
}[f(y_{1},y_{2})-f(x_{1},x_{2})]\alpha_{1}(x_{1},dy_{1})\alpha_{2}%
(x_{2},dy_{2})\\
&  ~~~~~~~~~~~~~~~~~+ag(x_{1},x_{2})[f(x_{2},x_{1})-f(x_{1},x_{2})].
\end{align*}
It is not difficult to check that the stationary distribution of
$\boldsymbol{\bar{X}}^{a}$ remains $\mu$ and that $\mathcal{L}^{a}$ is
self-adjoint. As before, the large deviation rate function $I^{a}$ for the
occupation measure
\begin{equation}
\eta_{T}^{a}=\frac{1}{T}\int_{0}^{T}\delta_{\boldsymbol{X}^{a}(t)}\,dt
\label{eqn:etaTa}%
\end{equation}
can be explicitly identified. Indeed, for any probability measure $\nu$ on
$\mathbb{R}^{d}\times\mathbb{R}^{d}$ with $\nu\ll\mu$ and $\theta=d\nu/d\mu$
\[
I^{a}(\nu)=I^{0}(\nu)+aJ(\nu),
\]
where
\[
J(\nu)=\int_{\mathbb{R}^{d}\times\mathbb{R}^{d}}g(x_{1},x_{2})\ell\left(
\sqrt{\frac{\theta(x_{2},x_{1})}{\theta(x_{1},x_{2})}}\right)  \nu
(dx_{1},dx_{2}).
\]
Note that as before, $I^{a}$ is monotonically increasing with respect to $a$.
Since $J(\nu)\geq0$ with equality if and only if $\theta(x_{1},x_{2}%
)=\theta(x_{2},x_{1})$ $\nu$-a.e., we have
\begin{equation}
I^{\infty}(\nu)\doteq\lim_{a\rightarrow\infty}I^{a}(\nu)=\left\{
\begin{array}
[c]{cl}%
I^{0}(\nu) & \mbox{if }\theta(x_{1},x_{2})=\theta(x_{2},x_{1})\text{ }%
\nu\text{-a.s.,}\\
\infty & \text{otherwise.}%
\end{array}
\right.  \label{eqn:Iinfty}%
\end{equation}

\subsection{Infinite swapping limit}

\label{sec:inf_swap_MP}

The infinite swapping limit for $\boldsymbol{X}^{a}$ as $a\rightarrow\infty$
can be similarly obtained by considering the corresponding temperature swapped
processes. Since the times between jumps determined by $\alpha_{1}$ and
$\alpha_{2}$ are always exponential with rate one, the infinite swapping limit
$\boldsymbol{Y}^{\infty}=(Y_{1}^{\infty},Y_{2}^{\infty})$ is a pure jump
Markov process where jumps occur according to a Poisson process with rate one.
In other words,
\[
\boldsymbol{Y}^{\infty}(t)=\boldsymbol{\bar{Y}}^{\infty}(j),~~\mbox{for~}\sum
_{i=1}^{j}\tau_{i}\leq t<\sum_{i=1}^{j+1}\tau_{i},
\]
where $\boldsymbol{\bar{Y}}^{\infty}$ is the embedded Markov chain and
$\{\tau_{i}\}$ a sequence of iid exponential random variables with rate one.
Furthermore, the probability transition kernel for $\boldsymbol{\bar{Y}%
}^{\infty}$ is
\begin{align}
\lefteqn{P\{\boldsymbol{\bar{Y}}^{\infty}(j+1)\in(dz_{1},dz_{2})|\boldsymbol
{\bar{Y}}^{\infty}(j)=(y_{1},y_{2})\}}\label{eqn:MP_transition}\\
&  =\rho(y_{1},y_{2})\alpha_{1}(y_{1},dz_{1})\alpha_{2}(y_{2},dz_{2}%
)+\rho(y_{2},y_{1})\alpha_{2}(y_{1},dz_{1})\alpha_{1}(y_{2},dz_{2}),\nonumber
\end{align}
where the weight function $\rho$ is defined as in Theorem
\ref{thm:inf_swp_lim}. It is not difficult to argue that the stationary
distribution for $\boldsymbol{Y}^{\infty}$ is
\[
\bar{\mu}(dy_{1},dy_{2})=\frac{1}{2}[\pi(y_{1},y_{2})+\pi(y_{2},y_{1}%
)]\,dy_{1}dy_{2},
\]
and the weighted occupation measure
\begin{equation}
\eta_{T}^{\infty}=\frac{1}{T}\int_{0}^{T}\left[  \rho({Y}_{1}^{\infty}%
(t),{Y}_{2}^{\infty}(t))\delta_{({Y}_{1}^{\infty}(t),{Y}_{2}^{\infty}%
(t))}+\rho({Y}_{2}^{\infty}(t),{Y}_{1}^{\infty}(t))\delta_{({Y}_{2}^{\infty
}(t),{Y}_{1}^{\infty}(t))}\right]  dt \label{eqn:etaTinf}%
\end{equation}
converges to $\mu(dx_{1},dx_{2})=\pi(x_{1},x_{2})dx_{1}dx_{2}$ as
$T\rightarrow\infty$. It is obvious that the dynamics of the infinite swapping
limit are symmetric and instantaneously equilibrate the contribution from
$(Y_{1},Y_{2})$ and $(Y_{2},Y_{1})$ according to the invariant measure, owing
to the weight function $\rho$.

We have the following uniform large deviation principle result, which
justifies the superiority of infinite swapping model. Its proof is deferred to
the appendix. It should be noted that rate identification is not covered by
the existing literature, even in the case of a fixed swapping rate, due to the
pure jump nature of the process.

\begin{theorem}
\label{thm:LDP_MP} The occupation measure $\{\eta_{T}^{\infty}:T>0\}$
satisfies a large deviation principle with rate function $I^{\infty}$. More
generally, define the finite swapping model as in Subsection
\ref{sec:finite_swap}. Consider any sequence $\{a_{T}:T>0\}\subset
\lbrack0,\infty]$ such that $a_{T}\rightarrow\infty$ as $T\rightarrow\infty$,
and interpret $a_{T}<\infty$ to mean that $\eta_{T}^{a_{T}}$ is defined by
(\ref{eqn:etaTa}) with $a=a_{T}$, and $a_{T}=\infty$ to mean that $\eta
_{T}^{a_{T}}$ is defined by (\ref{eqn:etaTinf}). Then $\left\{  \eta
_{T}^{a_{T}}:T>0\right\}  $ satisfies an LDP with the rate function
$I^{\infty}$ defined in equation (\ref{eqn:Iinfty}).
\end{theorem}

\section{Discrete time process models}

\label{sec:discrete}

\subsection{Conventional parallel tempering algorithms}

\label{subsect:two_temp_model}

In the discrete time, multi-temperature algorithms that are actually
implemented, a swap is attempted after a deterministic or random number of
time steps, with a success probability of the form (\ref{eq:swap_rate}). The
two temperatures corresponding to particles for which a swap is attempted can
be chosen according to a deterministic or random schedule, and as noted
previously are usually adjacent since otherwise the success probability
(\ref{eq:swap_rate}) will be too small to allow efficient exchange of information.

As before it suffices to describe the algorithm in the setting of two
temperatures. As in Section \ref{sec:MP}, let $\alpha_{i}(x,dy)$ denote the
probability transition kernel for temperature $\tau_{i}$ whose stationary
distribution has a density $\pi_{i}$ for $i=1,2$. For now let $N=1/a$ be a
fixed positive integer that determines the frequency of swap attempts. Let
$\boldsymbol{\bar{X}}=\{(\bar{X}_{1}(j),\bar{X}_{2}(j)):j=0,1,\ldots\}$ denote
the state process. Then the evolution of the dynamics is as follows. For any
integer $k\geq1$ and $(k-1)(N+1)\leq j\leq k(N+1)-2$,
\[
P\{\boldsymbol{\bar{X}}(j+1)\in(dy_{1},dy_{2})|\boldsymbol{\bar{X}}%
(j)=(x_{1},x_{2})\}=\alpha_{1}(x_{1},dy_{1})\alpha_{2}(x_{2},dy_{2})
\]
and for $j=k(N+1)-1$,
\begin{align*}
P\{\boldsymbol{\bar{X}}(j+1)=(x_{2},x_{1})|\boldsymbol{\bar{X}}(j)=(x_{1}%
,x_{2})\}  &  =g(x_{1},x_{2}),\\
P\{\boldsymbol{\bar{X}}(j+1)=(x_{1},x_{2})|\boldsymbol{\bar{X}}(j)=(x_{1}%
,x_{2})\}  &  =1-g(x_{1},x_{2}).
\end{align*}
Thus a swap is attempted after every $N$ ordinary time steps based on the
underlying transition kernels $\alpha_{1}$ and $\alpha_{2}$. The case $N=1/a$
with $a$ an integer greater than one corresponds to the case where multiple
swaps are attempted between two ordinary time steps. The unique invariant
distribution of $\bar{X}$ is $\mu(dx_{1}dx_{2})=\pi(x_{1},x_{2})\,dx_{1}%
dx_{2}$, regardless of the value of $N$, and the occupation measure
\[
\frac{1}{J}\sum_{j=0}^{J-1}\delta_{\boldsymbol{X}(j)}
\]
converges to $\mu$ as $J\rightarrow\infty$ almost surely.

\begin{remark}
\label{remark:swap_rate_dis}\emph{\ Note that }$N$\emph{\ could be random. For
example, if }$N$\emph{\ is chosen to be a geometric random variable with mean
}$\lambda$\emph{, then }$\bar{X}$\emph{\ is exactly the embedded Markov chain
of the pure jump Markov process }$\bar{X}^{a}$\emph{\ with }$a=1/\lambda
$\emph{\ in Subsection \ref{sec:finite_swap}. }
\end{remark}

\subsection{Infinite swapping model}

\label{subsect:discrete_time_inf_swap}

As with the continuous time case, to produce a well-defined limit one must
consider the temperature swapped process and then consider the limit as
swapping frequency tends to infinity. It turns out that the limit is exactly
the embedded Markov chain for the pure jump Markov process in Subsection
\ref{sec:inf_swap_MP}. That is, the infinite swapping limit in discrete time
is a Markov chain $\boldsymbol{\bar{Y}}^{\infty}=\left\{  (\bar{Y}_{1}%
^{\infty}(j),\bar{Y}_{2}^{\infty}(j)):j=0,1,\ldots\right\}  $ with the
transition kernel
\begin{equation}
\rho(y_{1},y_{2})\alpha_{1}(y_{1},dz_{1})\alpha_{2}(y_{2},dz_{2})+\rho
(y_{2},y_{1})\alpha_{2}(y_{1},dz_{1})\alpha_{1}(y_{2},dz_{2}).
\label{eq:two_temp_inf_swap_dyn}%
\end{equation}
The corresponding weighted empirical measure is
\begin{equation}
\eta_{J}^{\infty}\doteq\frac{1}{J}\sum_{j=0}^{J-1}\left[  \rho(\bar{Y}%
_{1}^{\infty}(j),\bar{Y}_{2}^{\infty}(j))\delta_{(\bar{Y}_{1}^{\infty}%
(j),\bar{Y}_{2}^{\infty}(j))}+\rho(\bar{Y}_{2}^{\infty}(j),\bar{Y}_{1}%
^{\infty}(j))\delta_{(\bar{Y}_{2}^{\infty}(j),\bar{Y}_{1}^{\infty}%
(j))}\right]  . \label{eq:two_temp_inf_swap_meas}%
\end{equation}

The generalization to multiple temperatures is also straightforward. Suppose
that there are $K$ temperatures. Denote the infinite swapping limit process by
$\boldsymbol{\bar{Y}}=\{\boldsymbol{\bar{Y}}(j):j=0,1,\ldots\}$, which is a
Markov chain taking values in the space $(\mathbb{R}^{d})^{K}$. Given that the
current state of the chain is $\boldsymbol{\bar{Y}}^{\infty}(j)=\boldsymbol{y}%
=(y_{1},\ldots,y_{K})$, define as before the weights
\[
w\left(  \boldsymbol{y}\right)  \doteq\frac{\pi\left(  \boldsymbol{y}\right)
}{\sum_{\theta\in S_{K}}\pi\left(  \boldsymbol{y_{\theta}}\right)  }.
\]
Then the transition kernel of $\boldsymbol{\bar{Y}}^{\infty}$ is
\[
P(\boldsymbol{\bar{Y}}^{\infty}(j+1)\in d\boldsymbol{z}|\boldsymbol{\bar{Y}%
}^{\infty}(j)=\boldsymbol{y})=\sum_{\sigma\in S_{K}}w(\boldsymbol{y}_{\sigma
})\alpha_{1}(y_{\sigma(1)},dz_{\sigma(1)})\cdots\alpha_{K}(y_{\sigma
(K)},dz_{\sigma(K)}).
\]
The discrete time numerical approximation to the invariant distribution is%
\[
\eta_{J}^{\infty}\doteq\frac{1}{J}\sum_{j=0}^{J-1}\sum_{\sigma\in S_{K}%
}w(\boldsymbol{\bar{Y}}_{\sigma}^{\infty}(j))\delta_{\boldsymbol{\bar{Y}%
}_{\sigma}^{\infty}(j)}.
\]

\begin{remark}
\label{remark:unif_LDP_dis}\emph{\ It is not difficult to derive large
deviation principles for the discrete time finite swapping or infinite
swapping models. However, it remains an open question whether the rate
function is monotonic with respect to the swap rate (frequency). However, the
discrete time large deviation rate function can be obtained from that of the
continuous time pure jump Markov process models through the contraction
principle, and the two coincide in the limit as the transition kernels
}$\alpha_{i}$\emph{\ correspond to an infinitesimal time step for the
diffusion process (\ref{eq:two_temp_diff}). Hence the discrete time rate
function will be at least approximately monotone, and in this sense the
infinite swapping limit should (at least approximately) dominate all finite
swapping algorithms. This is supported by the data presented in Section
\ref{sec:numerical} and the much more extensive empirical study presented in
\cite{pladoldupliuwangub}. }
\end{remark}

\section{Partial infinite swapping}

As noted in Section \ref{Subsect:K_temp_cont_time}, the number of weights and
their calculation can become unwieldy for infinite swapping even when the
number of temperatures is moderate. In this section we construct algorithms
that maintain most of the benefit of the infinite swapping algorithm but at a
much lower computational cost. In the first subsection we describe the
infinite swapping limit models when only a subgroup of the permutations of the
particles (respectively, temperatures) are allowed by the prelimit particle
(respectively, temperature) swapped process. The computational complexity of
these limit models will be controlled by limiting the number of permutations
that communicate with each other through the swapping mechanism. The infinite
swapping models in this subsection will be called \textit{partial infinite
swapping models,} as opposed to the full infinite swapping models in the
previous section. The second subsection shows how such partial infinite
swapping schemes can be interwoven to approximate the full infinite swapping model.

\subsection{Partial infinite swapping models}

We consider subsets $A$ of $S_{K}$ with the property that $A$ is an algebraic
subgroup of $S_{K}$. That is,

\begin{enumerate}
\item the identity belongs to $A$;

\item if $\sigma_{1}, \sigma_{2}\in A$ then $\sigma_{1}\circ\sigma_{2}\in A$,
where $\circ$ denotes composition;

\item if $\sigma\in A$ then $\sigma^{-1}\in A$.
\end{enumerate}

Although one can write down a partial infinite swapping model that corresponds
to instantaneous equilibration for an arbitrary subset $A$, it is only when
$A$ is a subgroup that the corresponding partial infinite swapping process has
an interpretation as the limit of parallel tempering type processes. When
alternating between partial infinite swapping processes, a \textquotedblleft
handoff\textquotedblright\ rule will be needed, and it is only for those which
correspond to subgroups that such a handoff rule is well defined. This point
is discussed in some detail in the next section.

The definition of the partial infinite swapping process based on $A$ is
completely analogous to that of the full infinite swapping process. The state
process $\{\boldsymbol{\bar{Y}}(j):j=0,1,\ldots\}$ is a Markov chain with the
transition kernel
\begin{equation}
\alpha^{A}(\boldsymbol{y},d\boldsymbol{z})\doteq\sum_{\sigma\in A}\tilde
{w}^{A}\left(  \boldsymbol{y}_{\sigma}\right)  \alpha_{1}(y_{\sigma
(1)},dz_{\sigma(1)})\cdots\alpha_{K}(y_{\sigma(K)},dz_{\sigma(K)})
\label{eqn:part_inf_swap}%
\end{equation}
and the weighted empirical measure is
\[
\tilde{\eta}_{J}\doteq\frac{1}{J}\sum_{j=0}^{J-1}\sum_{\sigma\in A}\tilde
{w}^{A}(\boldsymbol{\bar{Y}}_{\sigma}(j))\delta_{\boldsymbol{\bar{Y}}_{\sigma
}(j)},
\]
where the weight function $\tilde{w}^{A}$ is defined by
\begin{equation}
\tilde{w}^{A}\left(  \boldsymbol{y}\right)  \doteq\frac{\pi\left(
\boldsymbol{y}\right)  }{\sum_{\theta\in A}\pi\left(  \boldsymbol{y}_{\theta
}\right)  }, \label{eqn:part_weight}%
\end{equation}
and satisfies for any $\boldsymbol{y}$
\[
\sum_{\sigma\in A}\tilde{w}^{A}(\boldsymbol{y}_{\sigma})=1.
\]
We omit the dependence on both $a=\infty$ and $A$ from the notation. Note that
in contrast with the full swapping system, it is only those permutations of
$\boldsymbol{y}$ corresponding to $\sigma\in A$ that are balanced according to
the invariant distribution in their contributions to $\tilde{\eta}_{J}$.

To illustrate the construction we present a few examples. With a standard
abuse of notation denote the permutation $\sigma$ such that $\sigma(i)=a_{i}$
by the form $(a_{1},a_{2},\ldots,a_{K})$. In particular, $(1,2,\ldots,K)$ is
the identity of the group $S_{K}$.

\begin{example}
Let $K=4$ and $A=\{(1,2,3,4),(2,1,3,4)\}$. This corresponds to only allowing
swaps between temperatures $\tau_{1}$ and $\tau_{2}$ at the prelimit. Define
\[
\tilde{w}(\boldsymbol{y})=\frac{\pi(y_{1},y_{2},y_{3},y_{4})}{\pi(y_{1}%
,y_{2},y_{3},y_{4})+\pi(y_{2},y_{1},y_{3},y_{4})}
\]
The probability transition kernel of the corresponding partial infinite
swapping process is given by%
\begin{align*}
\lefteqn{\tilde{w}\left(  y_{1},y_{2},y_{3},y_{4}\right)  \alpha_{1}%
(y_{1},dz_{1})\alpha_{2}(y_{2},dz_{2})\alpha_{3}(y_{3},dz_{3})\alpha_{4}%
(y_{4},dz_{4})}\\
&  ~~~+\tilde{w}\left(  y_{2},y_{1},y_{3},y_{4}\right)  \alpha_{1}%
(y_{2},dz_{2})\alpha_{2}(y_{1},dz_{1})\alpha_{3}(y_{3},dz_{3})\alpha_{4}%
(y_{4},dz_{4})
\end{align*}
and the contribution to the weighted empirical measure is
\[
\tilde{w}\left(  \bar{Y}_{1},\bar{Y}_{2},\bar{Y}_{3},\bar{Y}_{4}\right)
\delta_{\left(  \bar{Y}_{1},\bar{Y}_{2},\bar{Y}_{3},\bar{Y}_{4}\right)
}+\tilde{w}\left(  \bar{Y}_{2},\bar{Y}_{1},\bar{Y}_{3},\bar{Y}_{4}\right)
\delta_{\left(  \bar{Y}_{2},\bar{Y}_{1},\bar{Y}_{3},\bar{Y}_{4}\right)  }.
\]
Note that with $\pi_{ij}$ denoting the marginal invariant distribution on the
$i$-th and $j$-th components, the weight function can be written as
\[
\tilde{w}(\boldsymbol{y})=\frac{\pi_{12}(y_{1},y_{2})}{\pi_{12}(y_{1}%
,y_{2})+\pi_{12}(y_{2},y_{1})},
\]
which is consistent with the weights of the two-temperature model in Theorem
\ref{thm:inf_swp_lim}.
\end{example}

\begin{example}
Again take $K=4$, but this time use the subgroup generated by $(2,1,3,4)$ and
$(1,2,4,3)$, i.e., $A=\{(1,2,3,4),(2,1,3,4),(1,2,4,3),(2,1,4,3)\}$. Then the
dynamics are given by
\begin{align*}
\lefteqn{ \tilde{w}\left(  y_{1},y_{2},y_{3},y_{4}\right)  \alpha_{1}%
(y_{1},dz_{1})\alpha_{2}(y_{2},dz_{2}) \alpha_{3}(y_{3},dz_{3})\alpha
_{4}(y_{4},dz_{4})}\\
&  ~~~ + \tilde{w}\left(  y_{2},y_{1},y_{3},y_{4}\right)  \alpha_{1}%
(y_{2},dz_{2})\alpha_{2}(y_{1},dz_{1}) \alpha_{3}(y_{3},dz_{3})\alpha
_{4}(y_{4},dz_{4})\\
&  ~~~~~~~~~+ \tilde{w}\left(  y_{1},y_{2},y_{4},y_{3}\right)  \alpha
_{1}(y_{1},dz_{1})\alpha_{2}(y_{2},dz_{2}) \alpha_{3}(y_{4},dz_{4})\alpha
_{4}(y_{3},dz_{3})\\
&  ~~~~~~~~~~~~~~~~+ \tilde{w}\left(  y_{2},y_{1},y_{4},y_{3}\right)
\alpha_{1}(y_{2},dz_{2})\alpha_{2}(y_{1},dz_{1}) \alpha_{3}(y_{4}%
,dz_{4})\alpha_{4}(y_{3},dz_{3})
\end{align*}
where the weight function $\tilde{w}$ is defined by
\[
\widetilde{w}\left(  \boldsymbol{y}\right)  =\frac{\pi\left(  y_{1}%
,y_{2},y_{3},y_{4}\right)  }{\pi\left(  y_{1},y_{2}, y_{3},y_{4}\right)
+\pi\left(  y_{2},y_{1}, y_{3},y_{4}\right)  +\pi\left(  y_{1},y_{2},
y_{4},y_{3}\right)  +\pi\left(  y_{2},y_{1}, y_{4},y_{3}\right)  }.
\]
The contribution to the weighted empirical measure is
\begin{align*}
\lefteqn{\tilde{w}\left(  \bar{Y}_{1},\bar{Y}_{2},\bar{Y}_{3},\bar{Y}%
_{4}\right)  \delta_{\left(  \bar{Y}_{1},\bar{Y}_{2},\bar{Y}_{3},\bar{Y}%
_{4}\right)  } +\tilde{w}\left(  \bar{Y}_{2},\bar{Y}_{1},\bar{Y}_{3},\bar
{Y}_{4}\right)  \delta_{\left(  \bar{Y}_{2},\bar{Y}_{1},\bar{Y}_{3},\bar
{Y}_{4}\right)  }}\\
&  ~~~ +\tilde{w}\left(  \bar{Y}_{1},\bar{Y}_{2},\bar{Y}_{4},\bar{Y}%
_{3}\right)  \delta_{\left(  \bar{Y}_{1},\bar{Y}_{2},\bar{Y}_{4},\bar{Y}%
_{3}\right)  }+\tilde{w}\left(  \bar{Y}_{2},\bar{Y}_{1},\bar{Y}_{4},\bar
{Y}_{3}\right)  \delta_{\left(  \bar{Y}_{2},\bar{Y}_{1},\bar{Y}_{4},\bar
{Y}_{3}\right)  }.
\end{align*}

\end{example}

\begin{example}
We let $K=3$ and take $A$ to be the subgroup of $S_{K}$ generated by the
rotation $(2,3,1)$, i.e., $A=\{(1,2,3),(2,3,1),(3,1,2)\}$. Then the dynamics
are given by%
\begin{align*}
\lefteqn{ \tilde{w}\left(  y_{1},y_{2},y_{3}\right)  \alpha_{1}(y_{1}%
,dz_{1})\alpha_{2}(y_{2},dz_{2}) \alpha_{3}(y_{3},dz_{3})}\\
&  ~~~ + \tilde{w}\left(  y_{2},y_{3},y_{1}\right)  \alpha_{1}(y_{2}%
,dz_{2})\alpha_{2}(y_{3},dz_{3}) \alpha_{3}(y_{1},dz_{1})\\
&  ~~~~~~~~~+ \tilde{w}\left(  y_{3},y_{1},y_{2}\right)  \alpha_{1}%
(y_{3},dz_{3})\alpha_{2}(y_{1},dz_{1}) \alpha_{3}(y_{2},dz_{2})
\end{align*}
where
\[
\tilde{w}\left(  \boldsymbol{y}\right)  =\frac{\pi(y_{1},y_{2},y_{3})}%
{\pi(y_{1},y_{2},y_{3})+\pi\left(  y_{2},y_{3},y_{1}\right)  +\pi\left(
y_{3},y_{1},y_{2}\right)  }
\]
and the contribution to the weighted empirical measure is%
\[
\tilde{w}_{1}\left(  \bar{Y}_{1},\bar{Y}_{2},\bar{Y}_{3}\right)
\delta_{\left(  \bar{Y}_{1},\bar{Y}_{2},\bar{Y}_{3}\right)  }+\tilde{w}_{2}(
\bar{Y}_{2},\bar{Y}_{3},\bar{Y}_{1})\delta_{\left(  \bar{Y}_{2},\bar{Y}%
_{3},\bar{Y}_{1}\right)  }+\tilde{w}_{3}\left(  \bar{Y}_{3},\bar{Y}_{1}%
,\bar{Y}_{3}\right)  \delta_{\left(  \bar{Y}_{3},\bar{Y}_{1},\bar{Y}%
_{3}\right)  }.
\]

\end{example}

The first two examples would correspond to the infinite swapping limit of a
standard parallel tempering process, where swaps between only $1$ and 2 are
allowed in the first example, and swaps between 1 and 2 and swaps between 3
and 4 are allowed in the second. Note that the computational complexity does
not increase significantly between the first and second example. The third
example corresponds to a very different sort of prelimit process, in which
\textquotedblleft rotations\textquotedblright\ of the coordinates
$(y_{1},y_{2},y_{3})\rightarrow\left(  y_{2},y_{3},y_{1}\right)
\rightarrow\left(  y_{3},y_{1},y_{2}\right)  \rightarrow(y_{1},y_{2},y_{3})$
are allowed. One can devise a Metropolis type rule that allows such
\textquotedblleft swaps\textquotedblright\ and yields the indicated infinite
swapping system.

\subsection{Approximating full infinite swapping by partial swapping}

In this section we consider the issue of alternating between such partial
infinite swapping systems to approximate the full infinite swapping limit. Let
$A$ and $B$ be subgroups of $S_{K}$. $A$ and $B$ are said to \textit{generate}
$S_{K}$ if the smallest subgroup that contains $A$ and $B$ is $S_{K}$ itself.
Note that the total number of permutations in $A\cup B$ can be significantly
smaller than $K!$, the size of $S_{K}$. In fact, it is possible to construct
subgroups $A$ and $B$ that generate $S_{K}$ and that the total number of
permutations in $A\cup B$ is of order $K$. There is an obvious extension to
more than two subgroups.

\begin{example}
\label{Ex:sub_1}Let $K=4$ and let $A$ be generated by $\left\{
(2,1,3,4),(1,3,2,4)\right\}  $ and $B$ be generated by $\left\{
(1,3,2,4),(1,2,4,3)\right\}  $, respectively. Thus $A$ is the collection of 6
permutations that fix the last component and allow all rearrangements of the
first three, while $B$ fixes the first component and allows all rearrangements
of the last three. Then $A$ and $B$ generate $S_{K}$.
\end{example}

\begin{example}
\label{Ex:sub_2}Let $K=4$ and let $A$ and $B$ be subgroups generated by
$\left\{  (2,1,3,4)\right\}  $ and $\left\{  (2,3,4,1)\right\}  $,
respectively. In other words, $A=\{(1,2,3,4), (2,1,3,4)\}$ corresponds to only
allowing permutations between the first two components, while $B =
\{(1,2,3,4),(2,3,4,1),(3,4,1,2),(4,1,2,3)\}$ corresponds to cycling of the
four temperatures. Then $A$ and $B$ generate $S_{K}$.
\end{example}

To keep the computational cost controlled, one can approximate the full
infinite swapping model by alternating between partial infinite swapping
processes whose associated subgroups generate the whole group. However, one
must be careful in how the \textquotedblleft handoff\textquotedblright\ is
made when switching between different partial swapping models. It turns out
that one cannot simply switch between different partial infinite swapping
dynamics (i.e., transition kernels). Recall that in order to get a consistent
approximation to the desired target invariant distribution we do not use the
empirical measure generated by $\boldsymbol{\bar{Y}}$, but rather a carefully
constructed weighted empirical measure that works with several permutations of
$\boldsymbol{\bar{Y}}$. Simply switching the dynamics and weights will in fact
produce an algorithm that may not converge to the target distribution.

To see how one should design a handoff rule, note that if one considers a
collection of transition kernels each having the same invariant distribution
and alternates between them in a way that does not depend on the outcomes
prior to a switch, then the resulting empirical measure will in fact converge
to the common invariant distribution. This fact is used (at least implicitly)
in the parallel tempering algorithm itself, where one alternates the pair of
particles being considered for swapping according to deterministic or random
rules so long as the random rules do not rely on previously observed outcomes.

Now we use the fact that each partial infinite swapping model is a limit of
either a parallel tempering algorithm where only some pairs of particles are
considered for swapping, or some more general form of parallel tempering which
would allow groups of particles to simultaneously swap (according to an
appropriate Metropolis-type acceptance rule). An example in the earlier
category would be $A$ in Example \ref{Ex:sub_1}, which arises if only the
pairs corresponding to temperatures $\tau_{1},\tau_{2}$ and $\tau_{2},\tau
_{3}$ are allowed to swap, whereas an example in the latter category would be
$B$ in Example \ref{Ex:sub_2}, which corresponds to allowing the particle at
temperature $\tau_{i}$ to move to the location of the particle at temperature
$\tau_{i-1}$ (with $\tau_{0}=\tau_{4}$), and the reverse. Furthermore, each of
these partial infinite swapping models arises as a limit of transition kernels
of the corresponding temperature swapped processes which preserve the same
common invariant distribution. In taking the limit as the swap rate tends to
infinity, the correspondence between particle locations for a particle swapped
process and the \textquotedblleft instantaneously
equilibrated\textquotedblright\ temperature swapped process $\bar{Y}$ is lost.
However, one can construct a consistent algorithm by \textit{reconstructing
this correspondence}. In fact one should choose the particle location
according to the probabilities (under the invariant distribution) associated
with the various permutations in the subgroup. See Subsection
\ref{section:dis_approx} for more detailed discussion on the intuition behind
this approximation algorithm.

We next present an algorithm for alternating between two partial infinite
swapping dynamics. The restriction to two is for notational convenience only.
Suppose that the dynamics are indexed by the corresponding subgroups $A $ and
$B$, and that $n_{A}$ steps of subgroup $A$ are to be alternated with $n_{B}$
steps of subgroup $B$. For simplicity we do not describe a \textquotedblleft
burn-in\textquotedblright\ period. As in (\ref{eqn:part_inf_swap}) and
(\ref{eqn:part_weight}) we let $\alpha^{A}(\boldsymbol{y},d\boldsymbol{z})$
and $\tilde{w}^{A}(\boldsymbol{y})$ denote the transition kernel for $A$ and
the weights allocated to the permutation $\sigma\in A$, respectively, and
similarly for $B$.

\begin{algorithm}
\label{algo:disc} (Approximation to full infinite swapping)

\begin{enumerate}
\item Initialization: $\boldsymbol{\bar{X}}^{A}(0)=\boldsymbol{\bar{Y}%
}(0;0)\in(\mathbb{R}^{d})^{K},\ell=1$.

\item Loop $\ell$:

\begin{enumerate}
\item Initialization for $A$ dynamics: set $\boldsymbol{\bar{Y}}%
(\ell;0)=\boldsymbol{\bar{X}}^{A}(\ell-1)$.

\item Subgroup $A$ dynamics: update $\boldsymbol{\bar{Y}}(\ell;k),k=1,\ldots
,n_{A}$ according to the transition kernel $\alpha^{A}$, and add
\[
\sum_{\sigma\in A}\tilde{w}^{A}(\boldsymbol{\bar{Y}}_{\sigma}(\ell
;k))\delta_{\boldsymbol{\bar{Y}_{\sigma}}(\ell;k)}
\]
to the un-normalized empirical measure.

\item Reconstructing particle locations at the end of $A$ dynamics: Let
$\boldsymbol{\bar{X}}^{B}(\ell)$ be a random sample from the set
$\{\boldsymbol{\bar{Y}_{\sigma}}(\ell;n_{A}): \sigma\in A\}$ according to the
weights $\{\tilde{w}^{A}(\boldsymbol{\bar{Y}}_{\sigma}(\ell;n_{A})): \sigma\in
A\}$.

\item Initialization for $B$ dynamics: set $\boldsymbol{\bar{Y}}(\ell
;n_{A})=\boldsymbol{\bar{X}}^{B}(\ell)$.

\item Subgroup $B$ dynamics: update $\boldsymbol{\bar{Y}}(\ell;k),k=n_{A}%
+1,\ldots,n_{A}+n_{B}$ according to the transition kernel $\alpha^{B}$, and
add
\[
\sum_{\sigma\in B}\tilde{w}^{B}(\boldsymbol{\bar{Y}}_{\sigma}(\ell
;k))\delta_{\boldsymbol{\bar{Y}_{\sigma}}(\ell;k)}
\]
to the un-normalized empirical measure.

\item Reconstructing particle locations at the end of $B$ dynamics: Let
$\boldsymbol{\bar{X}}^{A}(\ell)$ be a random sample from the set
$\{\boldsymbol{\bar{Y}_{\sigma}}(\ell;n_{A}+n_{B}): \sigma\in B\}$ according
to the weights $\{\tilde{w}^{B}(\boldsymbol{\bar{Y}}_{\sigma}(\ell;n_{A}%
+n_{B})): \sigma\in B\}$.

\item Set $\ell= \ell+1$ and loop back to (a).
\end{enumerate}

\item Normalize the empirical measure.
\end{enumerate}
\end{algorithm}

\subsection{Discussions on the approximation}

\label{section:dis_approx}

In this section we further discuss the intuition underlying the handoff rule
between different partial infinite swapping dynamics and the approximation
algorithm of the previous section. We temporarily assume that the model is in
continuous time since the intuition is most transparent in this case.%

%TCIMACRO{\FRAME{ftbpFU}{3.6288in}{2.0946in}{0pt}{\Qcb{Approximation via
%partial infinite swapping}}{\Qlb{fig:partial}}{swapfig2.eps}%
%width 3.6288in;  height 2.0946in;  depth 0pt;  original-width 9.0001in;
%original-height 5.1673in;  cropleft "0";  croptop "1";  cropright "1";
%cropbottom "0";  filename 'swapfig2.eps';file-properties "XNPEU";}} }%
%BeginExpansion
\begin{figure}[ptb]%
\centering
\includegraphics[
height=2.0946in,
width=3.6288in
]%
{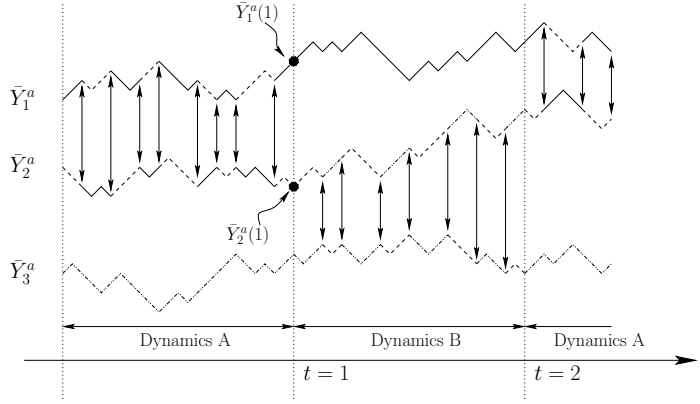}%
\caption{Approximation via partial infinite swapping}%
\label{fig:partial}%
\end{figure}
%EndExpansion

\noindent

For simplicity let us assume that there are three temperatures and two groups
$A=\{(1,2,3),(2,1,3)\}$ and $B=\{(1,2,3),(1,3,2)\}$. That is, under group $A$
dynamics only pairwise swaps between the temperatures $\tau_{1}$ and $\tau
_{2}$ are allowed, while under the group $B$ dynamics, only the swaps between
$\tau_{2}$ and $\tau_{3}$ are allowed. See Figure \ref{fig:partial}.

Consider the following prelimit swapping model. Let the swap rate be $a$. The
dynamics corresponding to group $A$ and group $B$ will be alternated on time
intervals of length $h$. Hence on the interval $[2kh,(2k+1)h)$ the particle
swapped process $(\bar{X}_{1}^{a},\bar{X}_{2}^{a},\bar{X}_{3}^{a})$ only
involves swaps between temperatures $\tau_{1}$ and $\tau_{2}$. One can easily
construct the corresponding temperature swapped process $(\bar{Y}_{1}^{a}%
,\bar{Y}_{2}^{a},\bar{Y}_{3}^{a})$ as before. Note that $\bar{X}_{3}^{a}%
=\bar{Y}_{3}^{a}$ on this time interval. Similarly, on the interval
$[(2k+1)h,(2k+2)h)$, only swaps between $\tau_{2}$ and $\tau_{3}$ are allowed
and on this interval $\bar{X}_{1}^{a}=\bar{Y}_{1}^{a}$. Note that there is no
ambiguity for the prelimit processes at the switch times $t=h,2h,\ldots,$
since the locations of the particles $(\bar{X}_{1}^{a},\bar{X}_{2}^{a},\bar
{X}_{3}^{a})$ are known.

Now consider the limit as $a\rightarrow\infty$ with $h$ being fixed. Without
loss of generality, we will only discuss how to deal with the switch of the
dynamics at time $t=h$. On the time interval $[0,h)$ we have the partial
infinite swapping limit process $\boldsymbol{\bar{Y}}^{A}=(\bar{Y}_{1},\bar
{Y}_{2},\bar{Y}_{3})$ that corresponds to the group $A$. Similarly it is clear
that on the time interval $[h,2h)$ we should have the partial infinite
swapping process $\boldsymbol{\bar{Y}}^{B}$ corresponding to the group $B$.
The problem is, however, by taking the limit, we lose the information on the
locations of the particles $(\bar{X}_{1}^{a},\bar{X}_{2}^{a},\bar{X}_{3}^{a})
$. Unless we can somehow recover this information at the switch time $t=h$ to
assign $\boldsymbol{\bar{Y}}^{B}(h)$, we cannot determine the dynamics of
$\boldsymbol{\bar{Y}}^{B}$ on $[h,2h)$. The key is to recall that the infinite
swapping limit instantaneously equilibrates multiple locations according to
the invariant distribution. In other words, given $\boldsymbol{\bar{Y}}%
^{A}(h-)=\boldsymbol{y}=(y_{1},y_{2},y_{3})$, the locations of the particles
are distributed according to
\[
\sum_{\sigma\in A}\tilde{w}^{A}(\boldsymbol{y}_{\sigma})\delta_{\boldsymbol{y}%
_{\sigma}}=\frac{\pi(y_{1},y_{2},y_{3})\delta_{(y_{1},y_{2},y_{3})}+\pi
(y_{2},y_{1},y_{3})\delta_{(y_{2},y_{1},y_{3})}}{\pi(y_{1},y_{2},y_{3}%
)+\pi(y_{2},y_{1},y_{3})}.
\]
Therefore, in order to identify the locations of the particles at time $h$, we
will take a random sample from this distribution once $\boldsymbol{\bar{Y}%
}^{A}(h-)$ is known. This explains the handoff rule used at the switch times
of partial infinite swapping processes.

Now we let $h\rightarrow0$. Since $A$ and $B$ generate the whole permutation
group $S_{K}$, it is easy to check that at each time instant, the locations of
$\{\boldsymbol{y}_{\sigma}:\sigma\in S_{K}\}$ are equilibrated according to
their invariant distribution, and therefore in the limit we will attain the
full infinite swapping model. This can be made rigorous by exploiting the time
scale separation between the slow diffusion processes $(\boldsymbol{\bar{Y}%
}^{A},\boldsymbol{\bar{Y}}^{B})$ and the fast switching process. We omit the
proof because the discussion is largely motivational.

Coming back to the discrete time partial infinite swapping model, it is clear
that Algorithm \ref{algo:disc} is nothing but a straightforward adaption of
the preceding discussion to discrete time. The only difference is that one
cannot establish an analogous result regarding approximation to the full
infinite swapping model as in continuous time. The subtlety here is that in
continuous time, as $h\rightarrow0$, one can basically ignore any effect from
the diffusion on any small time interval and assume that the process is only
making jumps between different permutations of a fixed triple $(y_{1}%
,y_{2},y_{3})$. This time scale separation is no longer valid in discrete
time. In this setting, the performance of a scheme based on interweaving
partial infinite swapping schemes lies between parallel tempering and full
infinite swapping, and computational results suggest that it is closer to the
latter than the former.

The issue of which interwoven partial schemes will perform best is an open
question. In practice we have used schemes of the following form. Suppose that
a set of say 45 temperatures is given. We then partition 45 into blocks of
sizes $3,6,\ldots,6$, with the first block containing the lowest three
temperatures, the second block the next six, and so on. Dynamic $A$ then is
given by allowing all permutations within each block. Note that the complexity
of the coefficients is then no worse than $6!$. In Dynamic $B$ we use the
partition $6,6,\ldots,6,3$. The form of the partial scheme is heuristically
motivated by allowing the largest possible overlap between the different
blocks when switching between dynamics, subject to the constraint that blocks
be of size no greater than $6$.

\section{Numerical examples}

\label{sec:numerical}In this section we present data comparing parallel
tempering at various swap rates and both full and partial infinite swapping.
We present what we call \textquotedblleft relaxation
studies.\textquotedblright\ The quantity of interest is the average potential
energy of the lowest temperature component under the invariant distribution.
In these studies, the system is run a long time to reach equilibrium, after
which it is repeatedly pushed out of equilibrium and we measure the time
needed to \textquotedblleft relax\textquotedblright\ back to equilibrium. Each
cycle consists of temporarily raising the temperatures of some of the lowest
temperature components for a number of steps sufficient to push the average
potential energy away from the \textquotedblleft true\textquotedblright\ value
(as measured by either sample or time averages). The temperatures are then
returned to their \textquotedblleft true\textquotedblright\ values for a fixed
number of steps, and the process is then repeated 2,000 times. We plot the
average of the 2,000 samples as a function of the number of moves, and the
performance of the algorithm is captured by the rate at which these averages
approach the correct value.%

%TCIMACRO{\TeXButton{TeX field}{\begin{multicols}{2}}}%
%BeginExpansion
\begin{multicols}{2}%
%EndExpansion%
%TCIMACRO{\FRAME{dtbpFU}{2.559in}{1.9813in}{0pt}{\Qcb{Figure 3}}{\Qlb{fig:3}%
%}{numfig1.eps}{\special{ language "Scientific Word";  type "GRAPHIC";
%width 2.559in;  height 1.9813in;  depth 0pt;  original-width 10.9745in;
%original-height 8.47in;  cropleft "0";  croptop "1";  cropright "1";
%cropbottom "0";  filename 'NumFig1.eps';file-properties "XNPEU";}} }%
%BeginExpansion
\begin{center}
\includegraphics[
height=1.9813in,
width=2.559in
]%
{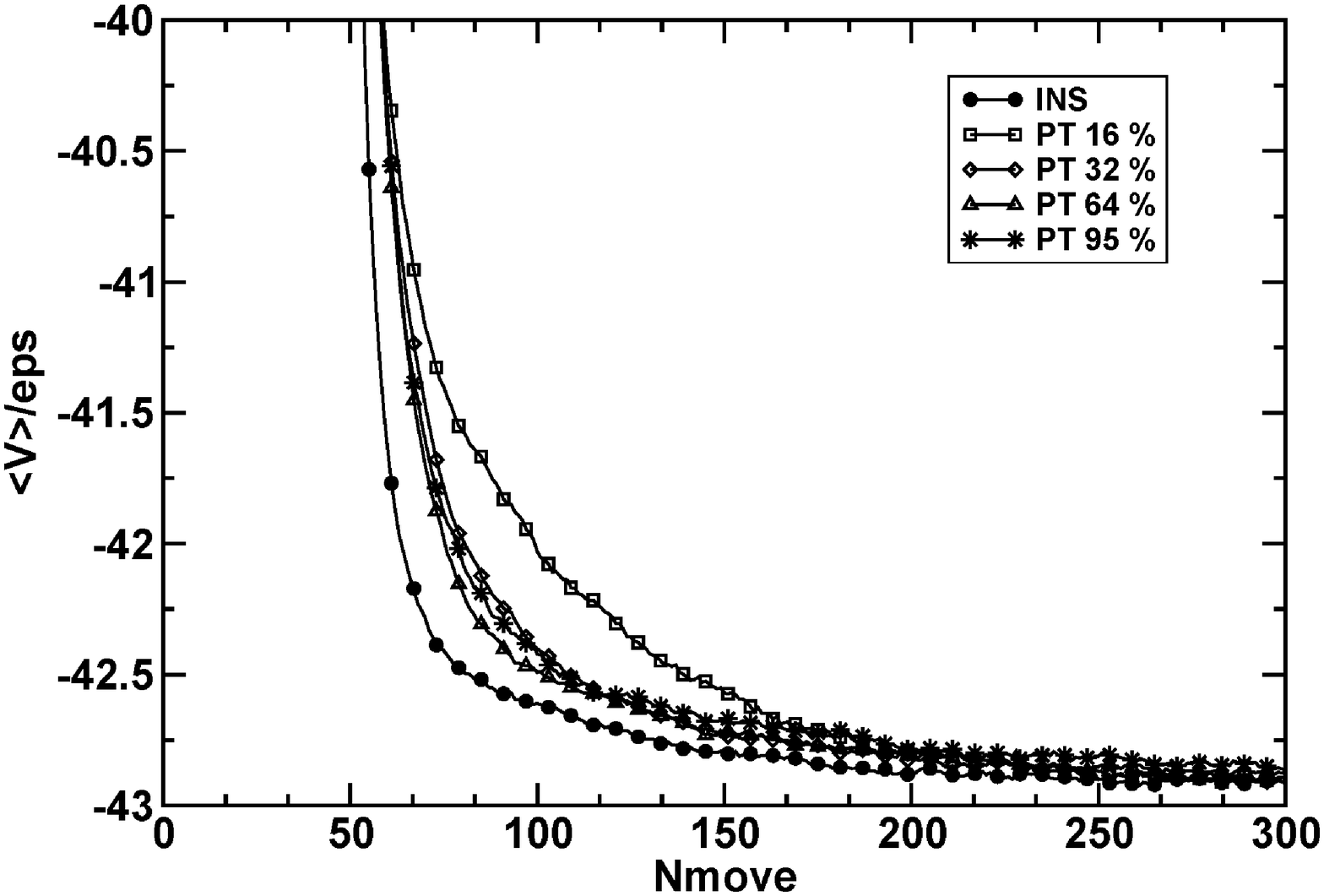}%
\\
Figure 3
\label{fig:3}%
\end{center}
%EndExpansion
%

%TCIMACRO{\FRAME{dtbpFU}{2.559in}{1.9813in}{0pt}{\Qcb{Figure 4}}{\Qlb{fig:4}%
%}{numfig2.eps}{\special{ language "Scientific Word";  type "GRAPHIC";
%width 2.559in;  height 1.9813in;  depth 0pt;  original-width 11.0004in;
%original-height 8.5002in;  cropleft "0";  croptop "1";  cropright "1";
%cropbottom "0";  filename 'NumFig2.eps';file-properties "XNPEU";}} }%
%BeginExpansion
\begin{center}
\includegraphics[
height=1.9813in,
width=2.559in
]%
{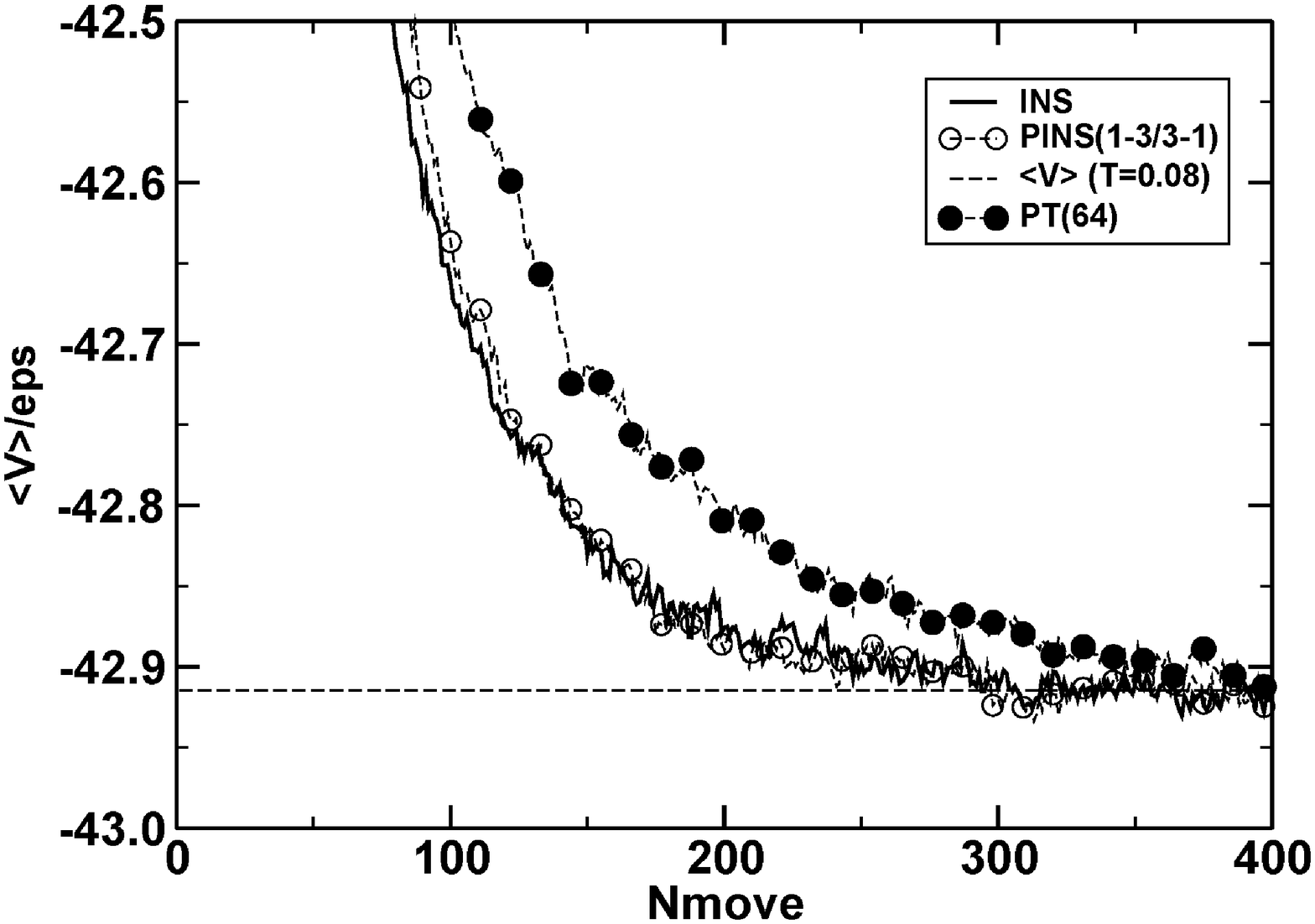}%
\\
Figure 4
\label{fig:4}%
\end{center}
%EndExpansion
%TCIMACRO{\TeXButton{TeX field}{\end{multicols}}}%
%BeginExpansion
\end{multicols}%
%EndExpansion

Figures 3 and 4 present data for a Lennard-Jones cluster of 13 atoms, using
the \textquotedblleft smart Monte Carlo\textquotedblright\ scheme of
\cite{rosdolfri} for the simulation of the dynamics, which produces a
relatively large move in configuration space for each step. The
\textquotedblleft true\textquotedblright\ value is approximately -42.92. This
is a relatively simple model, and was studied using only 4 temperatures. The
temperatures are dropped to the true values at step 50. Infinite swapping
converges more rapidly than any of the parallel tempering schemes. We see in
Figure 3 that the most efficient of the parallel tempering schemes appears to
use an attempted swap rate of around 64\%. [The rates that would typically be
used in such calculations are in the range of 5-10\%.] Figure 4 magnifies a
portion of the graph, but plots only the best parallel tempering result and
adds a partial infinite swapping result based on blocks of the form 1,3 and
3,1, and with a handoff at each Metropolis step. Little difference is observed
between the partial and full forms, though exclusive use of either of the
partial forms by itself performs poorly.%
%TCIMACRO{\FRAME{dtbpFU}{3.8778in}{3.0026in}{0pt}{\Qcb{Figure 5}}{\Qlb{fig:5}%
%}{numfig3.eps}{\special{ language "Scientific Word";  type "GRAPHIC";
%width 3.8778in;  height 3.0026in;  depth 0pt;  original-width 11.0004in;
%original-height 8.5002in;  cropleft "0";  croptop "1";  cropright "1";
%cropbottom "0";  filename 'NumFig3.eps';file-properties "XNPEU";}} }%
%BeginExpansion
\begin{center}
\includegraphics[
height=3.0026in,
width=3.8778in
]%
{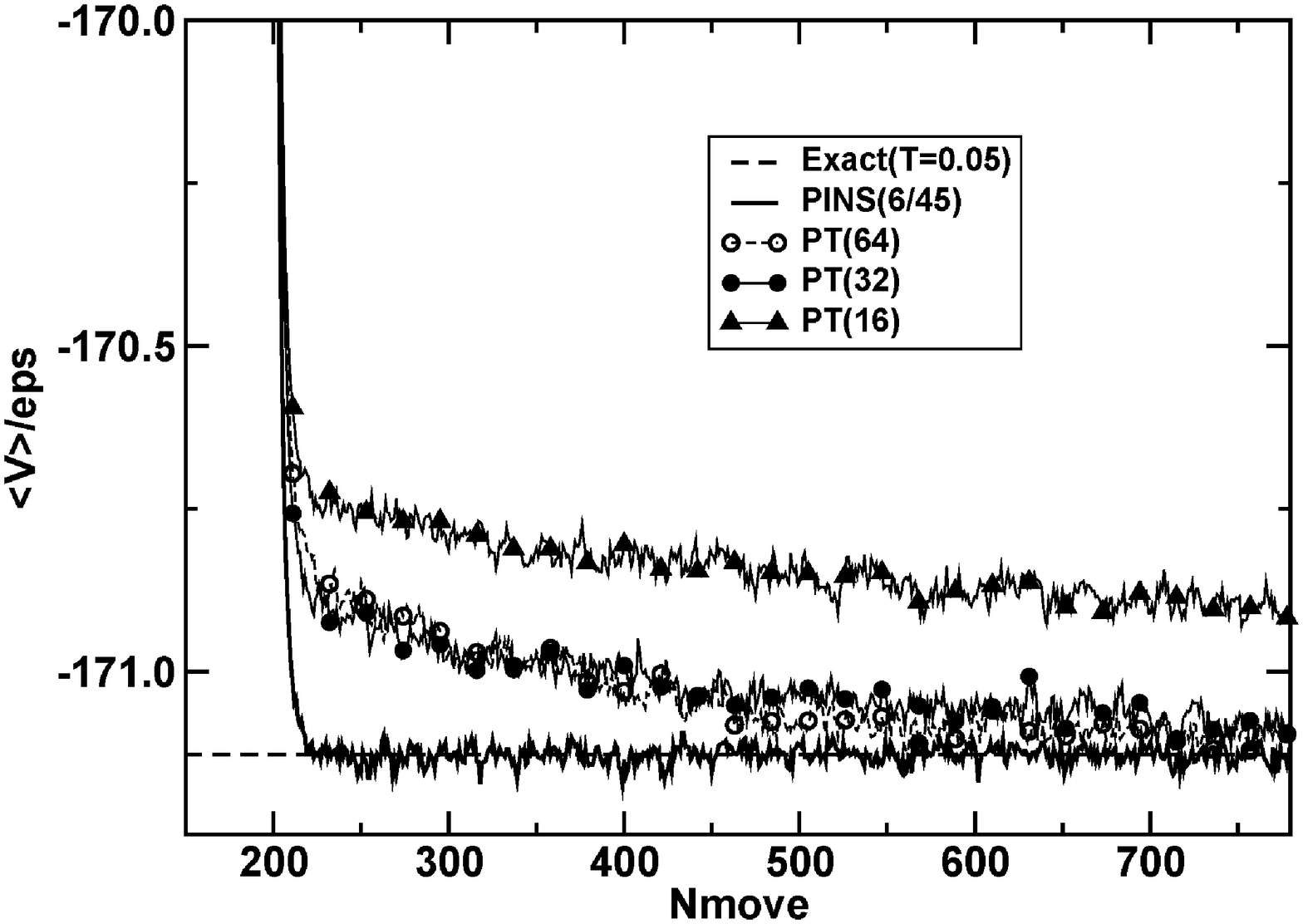}%
\\
Figure 5
\label{fig:5}%
\end{center}
%EndExpansion

The Lennard-Jones cluster of 13 atoms is not a particularly demanding problem,
but is presented so a comparison can be made between the full and partial
infinite swapping forms. A much more complex example is the Lennard-Jones
cluster of 38 atoms. Data for this example obtained using a 45 temperature
ensemble is given in Figure 5. Because full infinite swapping is impossible
for this larger computational ensemble, we use the partial form. For
comparison, results are also presented for parallel tempering.

The details concerning the computational methods underlying both the parallel
tempering and infinite swapping results of Figure 5 along with a discussion of
the temperature ensemble involved for this example can be found in
\cite{pladoldupliuwangub}. Briefly summarized, as with the previous 13-atom
Lennard-Jones example, Figure 5 denotes the results of a series of relaxation
experiments. Here, however, 45 temperatures are used with the lowest 15 being
involved in the heating/cooling process. The heating and cooling cycles
consist of 1200 smart Monte Carlo moves, each of one unit Lennard-Jones time
duration. The cooling segment is taken as the portion of the cycle from moves
200 to 800 with the remainder being the heating portion. During the cooling
portion of the cycle the 45 temperatures in the ensemble cover the range from
(0.050-0.210) in temperature steps of 0.005, and from (0.210 - 0.330) in steps
of 0.010 while during the heating portion of the cycle temperatures less than
or equal to 0.150 are set equal to 0.150. The results shown in Figure 5 are
obtained using 600 thermal cycles.

The 38-atom Lennard-Jones cluster has an interesting landscape. In particular,
while the global and lowest-lying local minima are similar in energy, the
minimum energy pathway that separates them involves appreciably higher
energies and contains 13 separate barriers \cite[Chapter 8.3]{wal2}. As
discussed in \cite{pladoldupliuwangub}, the partial infinite swapping approach
is appreciably more effective than conventional tempering approaches in
providing a proper sampling of this complex potential energy landscape.

\section{Appendix}

\subsection{Proof of Theorem \ref{thm:LDP_MP}}

Throughout the proof, we let $S=S_{1}\times S_{1}$, where $S_{1}%
\subset\mathbb{R}^{d}$ is convex and compact. Let $\mathcal{P}(S)$ denote the
Polish space of all probability measures on $S$ equipped with the topology of
weak convergence. For any probability measure $\nu\in\mathcal{P}(S)$, define
its mirror image $\nu^{R}\in\mathcal{P}(S)$ by requiring
\[
\nu^{R}(A\times B)=\nu(B\times A)
\]
for all Borel sets $A,B\subset\mathbb{R}^{d}$. Furthermore, as in the rest of
the paper, a bold symbol $\boldsymbol{x}\in S$ means $\boldsymbol{x}%
=(x_{1},x_{2})$, where $x_{1},x_{2}\in\mathbb{R}^{d}$, and $\boldsymbol{x}%
^{R}=$ $(x_{2},x_{1})$. We also use the notation
\[
\alpha(\boldsymbol{x},d\boldsymbol{y})\doteq\alpha_{1}(x_{1},dy_{1})\alpha
_{2}(x_{2},dy_{2}),
\]
which is a probability transition kernel defined on $S$ given $S$.

To prove the uniform large deviation principle, it suffices to prove the
equivalent uniform Laplace principle \cite[Chapter 1]{dupell4}. To simplify
the proof we have assumed that $S$ is compact. This would be the case if,
e.g., $V$ is defined with periodic boundary conditions. The general case can
be handled under (\ref{eq:growth_rate}) by using $V$ as a Lyapunov function
\cite[Section 8.2]{dupell4}. It will be convenient to split this into upper
and lower bounds. We also consider just the (more complicated) case where
$a_{T}\rightarrow\infty$ but $a_{T}<\infty$ for each $T$. Allowing
$a_{T}=\infty$ requires a different notation to handle this special case, but
does not change the structure of the proof otherwise.

We will show for any bounded continuous function $F:\mathcal{P}(S)\rightarrow
\mathbb{R}$ that
\begin{equation}
\lim_{T\rightarrow\infty}-\frac{1}{T}\log E\left[  \exp\{-TF(\eta_{T}^{a_{T}%
})\}\right]  =\inf_{\nu\in\mathcal{P}(S)}\left[  F(\nu)+I^{\infty}%
(\nu)\right]  . \label{eqn:Laplace}%
\end{equation}
By adding a constant to both sides we can assume $F\geq0$, and do so for the
rest of this section.

The proof of the uniform large deviation principle is based on the weak
convergence approach. The proof is complicated by the multiscale aspect of the
fast swapping process, as well as the fact that $\eta_{T}^{a_{T}}$ is a
weighted empirical measure that involves this fast process.

\subsection{Preliminary results}

\subsubsection{A representation}

We first state a stochastic control representation for the left hand side of
(\ref{eqn:Laplace}). As with the derivation of the infinite swapping process
via weak convergence, it will be necessary to work with the (distributionally
equivalent) temperature swapped processes for tightness to hold. In the
representation, all random variables used to construct $\eta_{T}^{a_{T}}$ are
replaced by random variables whose distribution is selected, and both the
distributions and the random variables will be distinguished from their
uncontrolled, original counterparts by an overbar. For this reason, while the
continuous time process is denoted by $\boldsymbol{Y}^{a}(t)$, we change
notation and use $\boldsymbol{U}^{a}(j)$ rather than $\boldsymbol{\bar{Y}}%
^{a}(j)$ to denote the discrete time process.

We first construct the temperature swapped process. Let $\alpha(\boldsymbol{x}%
,d\boldsymbol{y}|0)=\alpha(\boldsymbol{x},d\boldsymbol{y})$ and $\alpha
(\boldsymbol{x},d\boldsymbol{y}|1)=\alpha(\boldsymbol{x}^{R},d\boldsymbol{y}%
^{R})$, and let $\{N_{j}^{a},j=0,1,\ldots\}$ be iid geometric random variables
with parameter $1/(1+a)$, i.e. geometric random variables with mean $a$. Then
the random variables $\{\boldsymbol{U}^{a}(j),j=0,1,\ldots\}$, $\{M_{\ell}%
^{a}(j),j=0,1,\ldots,\ell=0,1,\ldots,N_{j}^{a}\}$ are constructed recursively
as follows. Given $M_{0}^{a}(j)=z$ and $\boldsymbol{U}^{a}(j)=\boldsymbol{x}$,
$\boldsymbol{U}^{a}(j+1)$ is distributed according to $\alpha(\boldsymbol{x}%
,d\boldsymbol{y}|z)$. The process $M_{\ell}^{a}(j),\ell=0,1,\ldots,N_{j}^{a}$
is a Markov chain with states $\left\{  0,1\right\}  $ and transition
probabilities%
\begin{equation}%
\begin{array}
[c]{cc}%
p(0,0|\boldsymbol{x})=g(\boldsymbol{x}) & p(0,1|\boldsymbol{x}%
)=1-g(\boldsymbol{x})\\
p(1,0|\boldsymbol{x})=1-g(\boldsymbol{x}^{R}) & p(1,0|\boldsymbol{x}%
)=g(\boldsymbol{x}^{R})
\end{array}
. \label{eqn:trans_probs}%
\end{equation}
The initial value for the subsequent interval is given by $M_{0}%
^{a}(j+1)=M_{N_{j}^{a}}^{a}(j)$. Letting $\{\tau_{j,\ell}^{a},i=0,1,\ldots
,\ell=0,1,\ldots,N_{j}^{a}-1\}$ be iid exponential random variables with mean
$1/a$, the temperature swapped process in continuous time is then given by%
\[
\boldsymbol{Y}^{a}(t)=\boldsymbol{U}^{a}(j)\text{ for }\sum_{i=0}^{j-1}%
\sum_{\ell=0}^{N_{j}^{a}-1}\tau_{i,\ell}^{a}\leq t<\sum_{i=0}^{j}\sum_{\ell
=0}^{N_{j}^{a}-1}\tau_{i,\ell}^{a},
\]
(with the convention that the sum from 0 to $-1$ is 0),
\[
Z^{a}(t)=M_{\ell}^{a}(j)\text{ for }\sum_{i=0}^{j-1}\sum_{k=0}^{\ell-1}%
\tau_{i,k}^{a}\leq t<\sum_{i=0}^{j-1}\sum_{k=0}^{\ell}\tau_{i,k}^{a},
\]
and lastly the ordinary and weighted empirical measures are given by%
\[
\psi_{T}^{a}=\frac{1}{T}\int_{0}^{T}\delta_{\boldsymbol{Y}^{a}(t)}dt\text{ and
}\eta_{T}^{a}=\frac{1}{T}\int_{0}^{T}\left[  1_{\left\{  Z^{a}(t)=0\right\}
}\delta_{\boldsymbol{Y}^{a}(t)}+1_{\left\{  Z^{a}(t)=1\right\}  }%
\delta_{\boldsymbol{Y}^{a}(t)^{R}}\right]  dt.
\]

Let $\sigma^{a}$ denote the exponential distribution with mean $1/a$ and let
$\beta^{a}$ denote the geometric distribution with mean $a$. For the
representation, all distributions [e.g., $\alpha(\boldsymbol{x}%
,d\boldsymbol{y}|z)$], can be perturbed from their original form, but such a
perturbation pays a relative entropy cost. We distinguish the new
distributions and random variables by using an overbar. Given $T\in(0,\infty
)$, let $R^{a}$ and $K^{a}$ be the discrete time indices when the continuous
time parameter reaches $T$, i.e.,
\begin{equation}
\sum_{i=0}^{R^{a}-2}\sum_{k=0}^{N_{i}^{a}-1}\tau_{i,k}^{a}+\sum_{k=0}%
^{K^{a}-1}\tau_{R^{a}-1,k}^{a}\leq T<\sum_{i=0}^{R^{a}-2}\sum_{k=0}^{N_{i}%
^{a}-1}\tau_{i,k}^{a}+\sum_{k=0}^{K^{a}}\tau_{R^{a}-1,k}^{a}.
\label{eqn:discrete_cont}%
\end{equation}

In this representation the barred quantities are constructed analogously to
their unbarred counterparts. Thus, e.g., $\bar{R}^{a}$ and $\bar{N}_{i}^{a}$
are defined by (\ref{eqn:discrete_cont}) but with $\tau_{i,k}^{a}$ replaced by
$\bar{\tau}_{i,k}^{a}$. Random variables corresponding to any given value of
$j$ are constructed in the order $\boldsymbol{\bar{U}}^{a}(j+1),\bar{N}%
_{j}^{a},\bar{M}_{\ell}^{a}(j),\bar{\tau}_{i,\ell}^{a},\ell=0,1,\ldots,\bar
{N}_{j}^{a}$, and then $j$ is updated to $j+1$. Barred measures, which are
also allowed to depend on discrete time, are used to construct the
corresponding barred random variables, e.g., $\boldsymbol{\bar{U}}^{a}(j+1)$
is (conditionally) distributed according to $\bar{\alpha}_{j}(\boldsymbol{\bar
{U}}^{a}(j),\cdot|\bar{M}_{0}^{a}(j))$. The infimum is over all collections of
measures $\{\bar{\alpha}_{j},\bar{\beta}_{j}^{a},\bar{p}_{j,\ell},\bar{\sigma
}_{j,\ell}^{a}\}$ and, although this is not denoted explicitly, any particular
measure can depend on all previously constructed random variables. To simplify
notation we let $\bar{N}_{\bar{R}^{a}}^{a}$ denote $\bar{K}^{a}$. We state the
representation for $\left\{  \eta_{T}^{a}\right\}  $, and note that an
analogous representation holds for $\left\{  \psi_{T}^{a}\right\}  $.

\begin{lemma}
\label{lem:rep}Let $G:\mathcal{P}(S)\times\mathbb{N}\rightarrow\mathbb{R}$ be
bounded from below and measurable. Then the representation%
\begin{align}
-  &  \frac{1}{T}\log E\left[  \exp\{-TG(\eta_{T}^{a},R^{a})\}\right]  =\inf
E\left[  \rule{0pt}{24pt}\rule{0pt}{18pt}\rule{0pt}{19pt}\rule{0pt}{18pt}%
G(\bar{\eta}_{T}^{a},\bar{R}^{a})\right. \nonumber\label{eqn:RE_rep}\\
&  \quad+\frac{1}{T}\sum_{i=0}^{\bar{R}^{a}-1}\left[  R\left(  \bar{\alpha
}_{i}(\boldsymbol{\bar{U}}^{a}(i),\cdot|\bar{M}_{0}^{a}(i))\left\Vert
\alpha(\boldsymbol{\bar{U}}^{a}(i),\cdot|\bar{M}_{0}^{a}(i))\right.  \right)
+R\left(  \bar{\beta}_{i}^{a}\left\Vert \beta^{a}\right.  \right)  \right]
\nonumber\\
&  \quad+\left.  \frac{1}{T}\sum_{i=0}^{\bar{R}^{a}-1}\sum_{k=0}^{\bar{N}%
_{i}^{a}-1}\left[  R\left(  \bar{p}_{i,k}(\bar{M}_{k}^{a}(i),\boldsymbol{\cdot
}|\boldsymbol{\bar{U}}^{a}(i))\left\Vert p(\bar{M}_{k}^{a}%
(i),\boldsymbol{\cdot}|\boldsymbol{\bar{U}}^{a}(i))\right.  \right)  +R\left(
\bar{\sigma}_{i,k}^{a}\left\Vert \sigma^{a}\right.  \right)  \right]
\rule{0pt}{18pt}\rule{0pt}{15pt}\right]  \rule{0pt}{18pt}\rule{0pt}{18pt}%
\rule{0pt}{19pt}\rule{0pt}{21pt}\rule{0pt}{23pt}\nonumber
\end{align}
is valid.
\end{lemma}

The proof of such representations follow from the chain rule for relative
entropy (see, e.g., \cite[Section B.2]{dupell4}). A novel feature of the
representation here is that the total number of discrete time steps is random.
However, this case can easily be reduced to the case with a fixed
deterministic number of steps.

\subsubsection{Rate for the ordinary empirical measure}

\textbf{Notation for marginals.} We will frequently factor measures on product
spaces in the proof. For a (deterministic) probability measure $\nu$ on a
product space such as $S^{1}\times S^{2}\times S^{3}$, with each $S^{i}$ a
Polish space, we use notation such as $\nu_{1,2}$ to denote the marginal
distribution on the first 2 components, and notation such as $\nu_{1|3}$ to
denote the conditional distribution on the first component given the third.
When $\nu$ is a measurable random measure these can all be chosen so that they
are also measurable.

We will make use of the rate function for the ordinary empirical measure. Let
$\varphi(\boldsymbol{x},d\boldsymbol{y})=\alpha(\boldsymbol{x},d\boldsymbol{y}%
|0)\rho_{0}(\boldsymbol{x})+\alpha(\boldsymbol{x},d\boldsymbol{y}|1)\rho
_{1}(\boldsymbol{x})$, where $\rho_{0}(\boldsymbol{x})=\rho(\boldsymbol{x})$
and $\rho_{1}(\boldsymbol{x})=\rho(\boldsymbol{x}^{R})$, and let $\bar{\mu
}(d\boldsymbol{x})=[\pi(\boldsymbol{x})+\pi(\boldsymbol{x}^{R}%
)]d\boldsymbol{x}/2$ be its unique invariant probability distribution. If
$\gamma$ is absolutely continuous with respect to $\bar{\mu}$ with
$\kappa(\boldsymbol{x})=[d\gamma/d\bar{\mu}](\boldsymbol{x})$, then set%
\[
K(\gamma)=1-\int\sqrt{\kappa(\boldsymbol{x})\kappa(\boldsymbol{y})}\bar{\mu
}(d\boldsymbol{x})\varphi(\boldsymbol{x},d\boldsymbol{y}).
\]
Note that $K$ is convex. We then extend the definition to all of
$\mathcal{P}(S)$ via lower semicontinuous regularization with respect to the
weak topology. Thus if $\gamma_{i}\rightarrow\gamma$ in the weak topology and
if each $\gamma_{i}$ is absolutely continuous with respect to $\bar{\mu}$,
then $\liminf_{i}K(\gamma_{i})\geq K(\gamma)$, and we have equality for at
least one such sequence. Note that since $\bar{\mu}$ is mutually absolutely
continuous with respect to Lebesgue measure, this means that $K(\gamma)\leq1$
for all $\gamma\in\mathcal{P}(S)$.

The following lemma will help relate weak limits of quantities in the
representations to the rate function of the ordinary empirical measure.

\begin{lemma}
\label{lem:rate_relation}Let $\gamma\in\mathcal{P}(S)$ be absolutely
continuous with respect to $\bar{\mu}$ and let $\kappa=[d\gamma/d\bar{\mu}]$.
Assume $A\in(0,\infty),\nu\in\mathcal{P}(S\times S)$ is such that $[\nu
]_{1}=[\nu]_{2}$, $R(\nu\Vert\lbrack\nu]_{1}\otimes\varphi)<\infty$, $r$ is
such that $r[\nu]_{1}=\kappa\bar{\mu}$, and $0\leq-\int\log r(\boldsymbol{y}%
)[\nu]_{1}(d\boldsymbol{y})<\infty$. Then
\begin{align}
K(\gamma)  &  =1-\int\sqrt{\kappa(\boldsymbol{x})\kappa(\boldsymbol{y})}%
\bar{\mu}(d\boldsymbol{x})\varphi(\boldsymbol{x},d\boldsymbol{y}%
)\label{eqn:two_rates}\\
&  \leq AR(\nu\Vert\lbrack\nu]_{1}\otimes\varphi)-A\int\log r(\boldsymbol{y}%
)[\nu]_{1}(d\boldsymbol{y})+A\log A-A+1.\nonumber
\end{align}

\end{lemma}

\begin{proof}
Let $C$ be the set where $\kappa(\boldsymbol{x})=0$. Then $-\int\log
r(\boldsymbol{y})[\nu]_{1}(d\boldsymbol{y})<\infty$ implies $r(\boldsymbol{y}%
)>0$ a.s. with respect to $[\nu]_{1}(d\boldsymbol{y})$, and we also have $\int
r(\boldsymbol{y})[\nu]_{1}(d\boldsymbol{y})=1$, so that $r(\boldsymbol{y}%
)<\infty$ a.s. with respect to $[\nu]_{1}(d\boldsymbol{y})$. It follows that
$[\nu]_{1}(C)=0$. Now suppose that
\[
\int\sqrt{\kappa(\boldsymbol{x})\kappa(\boldsymbol{y})}\bar{\mu}%
(d\boldsymbol{x})\varphi(\boldsymbol{x},d\boldsymbol{y})=0.
\]
Then $\kappa(\boldsymbol{x})\kappa(\boldsymbol{y})=0$ a.s. with respect to
Lebesgue measure, and so if $\kappa(\boldsymbol{x})\neq0$ then $\varphi
(\boldsymbol{x},\left\{  \boldsymbol{y}:\kappa(\boldsymbol{y})>0\right\}
)=0$, or $\varphi(\boldsymbol{x},C)=1$. Thus $\nu(\left(  S\backslash
C\right)  \times C)=0$, while $[\nu]_{1}\otimes\varphi(\left(  S\backslash
C\right)  \times C)=1$, which implies $R(\nu\Vert\lbrack\nu]_{1}\otimes
\varphi)=\infty$, which is a contradiction. We conclude that
\[
\int\sqrt{\kappa(\boldsymbol{x})\kappa(\boldsymbol{y})}\bar{\mu}%
(d\boldsymbol{x})\varphi(\boldsymbol{x},d\boldsymbol{y})>0.
\]
Since also
\[
\int\sqrt{\kappa(\boldsymbol{x})\kappa(\boldsymbol{y})}\bar{\mu}%
(d\boldsymbol{x})\varphi(\boldsymbol{x},d\boldsymbol{y})\leq1,
\]
it follows that
\[
-\log\int\sqrt{\kappa(\boldsymbol{x})\kappa(\boldsymbol{y})}\bar{\mu
}(d\boldsymbol{x})\varphi(\boldsymbol{x},d\boldsymbol{y})\in\lbrack0,\infty).
\]

Since $R(\nu\Vert\lbrack\nu]_{1}\otimes\varphi)<\infty$ and since $\bar{\mu}$
is the invariant distribution of $\varphi$ it follows that $R([\nu]_{1}%
\Vert\bar{\mu})<\infty$ \cite[Lemma 8.6.2]{dupell4}. This means that
\[
\int\log\frac{\kappa(\boldsymbol{x})}{r(\boldsymbol{x})}[\nu]_{1}%
(d\boldsymbol{x})<\infty,
\]
and since $-\int\log r(\boldsymbol{y})[\nu]_{1}(d\boldsymbol{y})<\infty$, it
follows that $-\int\log\kappa(\boldsymbol{x})[\nu]_{1}(d\boldsymbol{x}%
)>-\infty$. From $[\nu]_{1}=[\nu]_{2}$ we conclude that
\begin{equation}
-\frac{1}{2}\int\left[  \log\kappa(\boldsymbol{x})+\log\kappa(\boldsymbol{y}%
)\right]  \nu(d\boldsymbol{x},d\boldsymbol{y})>-\infty. \label{eqn:RE_LB}%
\end{equation}

By relative entropy duality (\cite[Proposition 1.4.2]{dupell4})
\begin{align*}
&  -\log\int\sqrt{\kappa(\boldsymbol{x})\kappa(\boldsymbol{y})}\bar{\mu
}(d\boldsymbol{x})\varphi(\boldsymbol{x},d\boldsymbol{y})\\
&  =-\log\int e^{\frac{1}{2}\left[  \log\kappa(\boldsymbol{x})+\log
\kappa(\boldsymbol{y})\right]  }\bar{\mu}(d\boldsymbol{x})\varphi
(\boldsymbol{x},d\boldsymbol{y})\\
&  \leq R(\nu\Vert\bar{\mu}\otimes\varphi)-\frac{1}{2}\int\left[  \log
\kappa(\boldsymbol{x})+\log\kappa(\boldsymbol{y})\right]  \nu(d\boldsymbol{x}%
,d\boldsymbol{y})
\end{align*}
is valid as long as the right hand side is not of the form $\infty-\infty$,
which is true by (\ref{eqn:RE_LB}). The chain rule then gives
\begin{align*}
&  -\log\int e^{\frac{1}{2}\left[  \log\kappa(\boldsymbol{x})+\log
\kappa(\boldsymbol{y})\right]  }\bar{\mu}(d\boldsymbol{x})\varphi
(\boldsymbol{x},d\boldsymbol{y})\\
&  \quad\leq R(\nu\Vert\bar{\mu}\otimes\varphi)-\int\log\kappa(\boldsymbol{x}%
)[\nu]_{1}(d\boldsymbol{x})\\
&  \quad=R(\nu\Vert\lbrack\nu]_{1}\otimes\varphi)+R([\nu]_{1}\Vert\bar{\mu
})-\int\log\kappa(\boldsymbol{x})[\nu]_{1}(d\boldsymbol{x})\\
&  \quad=R(\nu\Vert\lbrack\nu]_{1}\otimes\varphi)-\int\log r(\boldsymbol{x}%
)[\nu]_{1}(d\boldsymbol{x}),
\end{align*}
and thus%
\[
-\int\sqrt{\kappa(\boldsymbol{x})\kappa(\boldsymbol{y})}\bar{\mu
}(d\boldsymbol{x})\varphi(\boldsymbol{x},d\boldsymbol{y})\leq-e^{-R(\nu
\Vert\lbrack\nu]_{1}\otimes\varphi)+\int\log r(\boldsymbol{x})[\nu
]_{1}(d\boldsymbol{x})}.
\]
Then (\ref{eqn:two_rates}) follows from the fact that if $a\in\mathbb{R}$ and
$b\in(0,\infty)$ then $-e^{-a}\leq ab+b\log b-b$ by taking $a=R(\nu
\Vert\lbrack\nu]_{1}\otimes\varphi)-\int\log r(\boldsymbol{x})[\nu
]_{1}(d\boldsymbol{x})$ and $b=A$.
\end{proof}

\subsubsection{Decomposition of the exponential distribution}

In the construction of $\eta_{T}^{a}$ we used independent exponential random
variables $\tau_{i,k}^{a}$ and geometric random variables $N_{i}^{a}$, and the
fact that for each $i$
\[
\sum_{\ell=0}^{N_{i}^{a}-1}\tau_{i,\ell}^{a}
\]
is exponential with mean one. This deomposition corresponds to a relationship
in relative entropies, which we now state.

\begin{lemma}
\label{lem:RE_decomp}For $a\in(0,\infty)$, let $\bar{N}^{a}$ be distributed
according to a random probability measure $\bar{\beta}^{a}$ on $\left\{
0,1,\ldots\right\}  $. Given $\bar{N}^{a}=\ell$, for $k\in\left\{  0,1,\ldots
l\right\}  $ let $\bar{\tau}_{k}^{a}(\ell)$ be distributed according to a
random probability measure $\bar{\sigma}_{k}^{a}(\ell)$ on $[0,\infty)$. Let
$\bar{\sigma}$ be the distribution of the random variable
\[
\bar{\tau}=\sum_{k=0}^{\bar{N}^{a}-1}\bar{\tau}_{k}^{a}(\bar{N}^{a}).
\]
Then%
\[
E\left[  R\left(  \bar{\beta}^{a}\left\Vert \beta^{a}\right.  \right)
+\sum_{k=0}^{\bar{N}^{a}-1}R\left(  \bar{\sigma}_{k}^{a}(\bar{N}%
^{a})\left\Vert \sigma^{a}\right.  \right)  \right]  \geq E\left[  R\left(
\bar{\sigma}\left\Vert \sigma^{1}\right.  \right)  \right]  .
\]

\end{lemma}

\begin{proof}
Define a measure $\mu$ on the space $%
%TCIMACRO{\U{2115} }%
%BeginExpansion
\mathbb{N}
%EndExpansion
_{+}\times\prod_{i=0}^{\infty}[0,\infty)$ as follows. For any $\ell\in%
%TCIMACRO{\U{2115} }%
%BeginExpansion
\mathbb{N}
%EndExpansion
_{+}$ and any sequence $A_{0},A_{1},\ldots$ of Borel measurable subsets of
$[0,\infty)$,
\[
\mu\left(  \left\{  l\right\}  \times A_{0}\times A_{1}\times\ldots\right)
=\beta^{a}\left(  l\right)  \prod_{k=0}^{\infty}\sigma^{a}\left(
A_{k}\right)  ,
\]
and similarly define the measure $\bar{\mu}$ by%
\[
\bar{\mu}\left(  \left\{  l\right\}  \times A_{0}\times A_{1}\times
\ldots\right)  =\bar{\beta}^{a}\left(  l\right)  \left(  \prod_{k=0}^{\ell
-1}\bar{\sigma}_{k}^{a}(\ell)\left(  A_{k}\right)  \right)  \left(
\prod_{k=\ell}^{\infty}\sigma^{a}\left(  A_{k}\right)  \right)  .
\]
Then by the chain rule of relative entropy%
\[
E\left[  R\left(  \bar{\mu}\left\Vert \mu\right.  \right)  \right]  =E\left[
R\left(  \bar{\beta}^{a}\left\Vert \beta^{a}\right.  \right)  +\sum_{\ell
=0}^{\infty}\sum_{k=0}^{\ell-1}R\left(  \bar{\sigma}_{k}^{a}(\ell)\left\Vert
\sigma^{a}\right.  \right)  \bar{\beta}^{a}\left(  \ell\right)  \right]  .
\]
Since
\begin{align*}
E\left[  \sum_{k=0}^{\bar{N}^{a}-1}R\left(  \bar{\sigma}_{k}^{a}(\bar{N}%
^{a})\left\Vert \sigma^{a}\right.  \right)  \right]   &  =E\left[  E\left[
\left.  \sum_{k=0}^{\bar{N}^{a}-1}R\left(  \bar{\sigma}_{k}^{a}(\bar{N}%
^{a})\left\Vert \sigma^{a}\right.  \right)  \right\vert \bar{N}^{a}\right]
\right] \\
&  =E\left[  \sum_{\ell=0}^{\infty}\sum_{k=0}^{\ell-1}R\left(  \bar{\sigma
}_{k}^{a}(\ell)\left\Vert \sigma^{a}\right.  \right)  \bar{\beta}^{a}\left(
\ell\right)  \right]  ,
\end{align*}
it follows that
\[
E\left[  R\left(  \bar{\mu}\left\Vert \mu\right.  \right)  \right]  =E\left[
R\left(  \bar{\beta}^{a}\left\Vert \beta^{a}\right.  \right)  +\sum
_{k=0}^{\bar{N}^{a}-1}R\left(  \bar{\sigma}_{k}^{a}(\bar{N}^{a})\left\Vert
\sigma^{a}\right.  \right)  \right]  .
\]
Observe that $\bar{\tau}$ can be written as a measurable mapping on $%
%TCIMACRO{\U{2115} }%
%BeginExpansion
\mathbb{N}
%EndExpansion
_{+}\times\prod_{i=0}^{\infty}[0,\infty)$ and that $\bar{\sigma}$ and
$\sigma^{1}$ are the distributions induced on $[0,\infty)$ under that map by
$\bar{\mu}$ and $\mu$, respectively. Since relative entropy can only decrease
under such a mapping, it follows that $ER\left(  \bar{\mu}\left\Vert
\mu\right.  \right)  \geq ER\left(  \bar{\sigma}\left\Vert \sigma^{1}\right.
\right)  $.
\end{proof}

\subsection{Lower bound}

The proof of the lower bound will be partitioned into three cases according to
a parameter $C\in(1,\infty)$. After the three cases have been argued, the
proof of the lower bound will be completed. The first two cases are very
simple, and give estimates when $R^{a}/T$ is small (i.e., unusually few
exponential clocks with mean 1 are needed before time $T$ is reached) or when
$R^{a}/T$ is large (i.e., an unusually large number of of such clocks are
needed to reach tome $T$). The processes $\left\{  \boldsymbol{U}%
^{a}(j)\right\}  $ and $\left\{  M_{\ell}^{a}(j)\right\}  $ play no role in
these estimates, and the required estimates follow from Chebyshev's inequality
as it is used in the proof of Cr\'{a}mer's Theorem. We will need the function
$h\left(  b\right)  =-\log b+b-1$, $b\in\lbrack0,\infty)$, which satisfies
$\inf\left\{  R\left(  \gamma\left\Vert \sigma^{1}\right.  \right)  :\int
u\gamma\left(  du\right)  =b\right\}  =h\left(  b\right)  $.

\subsubsection{ The case $R^{a}/T\leq1/C.$}

Let $F:\mathcal{P}(S)\rightarrow\mathbb{R}$ be non-negative and continuous.
Then
\begin{align*}
-\frac{1}{T}  &  \log E\left[  \exp\{-TF(\eta_{T}^{a})-\infty1_{\left\{
[0,1/C]\right\}  ^{c}}(R^{a}/T)\}\right] \\
&  =-\frac{1}{T}\log E\left[  \exp\left\{  -TF(\eta_{T}^{a})\right\}
1_{\left\{  [0,1/C]\right\}  }(R^{a}/T)\right] \\
&  \geq-\frac{1}{T}\log P\left\{  \sum_{i=0}^{\left\lfloor T/C\right\rfloor
+1}\tau_{i}^{1}\geq T\right\}  .
\end{align*}
By Chebyshev's inequality, for $\alpha\in\left(  0,1\right)  $%
\begin{align*}
P\left\{  \sum_{i=0}^{\left\lfloor T/C\right\rfloor +1}\tau_{i}^{1}\geq
T\right\}   &  =P\left\{  e^{\alpha\sum_{i=0}^{\left\lfloor T/C\right\rfloor
+1}\tau_{i}^{1}}\geq e^{\alpha T}\right\} \\
&  \leq\exp\left\{  \left(  \left\lfloor T/C\right\rfloor +2\right)  \left(
\log\frac{1}{1-\alpha}-\alpha C\right)  \right\}  .
\end{align*}
Optimizing this inequality over $\alpha\in\left(  0,1\right)  $ gives
\begin{align*}
\liminf_{T\rightarrow\infty}  &  -\frac{1}{T}\log E\left[  \exp\{-TF(\eta
_{T}^{a})-\infty1_{\left\{  [0,1/C]\right\}  ^{c}}(R^{a}/T)\}\right] \\
&  \geq\liminf_{T\rightarrow\infty}\frac{1}{T}\left(  \left\lfloor
T/C\right\rfloor +2\right)  h\left(  C\right) \\
&  =\frac{h\left(  C\right)  }{C}.
\end{align*}
Note that $h\left(  C\right)  /C\rightarrow1$ as $C\rightarrow\infty$.

\subsubsection{ The case $R^{a}/T\geq C.$}

With $F$ as in the last section, an analogous argument gives
\begin{align*}
\liminf_{T\rightarrow\infty}  &  -\frac{1}{T}\log E\left[  \exp\{-TF(\eta
_{T}^{a})-\infty1_{\left\{  [C,\infty)\right\}  ^{c}}(R^{a}/T)\}\right] \\
&  \geq\liminf_{T\rightarrow\infty}\frac{1}{T}\left(  \left\lfloor T\left(
C-1\right)  \right\rfloor +2\right)  h\left(  \frac{1}{C-1}\right) \\
&  =\left(  C-1\right)  h\left(  \frac{1}{C-1}\right)  .
\end{align*}
Note that $\left(  C-1\right)  h\left(  1/\left(  C-1\right)  \right)
\rightarrow\infty$ as $C\rightarrow\infty$.\qquad.

\subsubsection{ The case $1/C\leq R^{a}/T\leq C.$}

To analyze this case it will be sufficient to consider any deterministic
sequence $\left\{  r^{a}\right\}  $ such that $1/C\leq r^{a}/T\leq C$, and
such that $r^{a}/T\rightarrow A$ as $T\rightarrow\infty$, and obtain lower
bounds on
\[
\liminf_{T\rightarrow\infty}-\frac{1}{T}\log E\left[  \exp\{-TF(\eta_{T}%
^{a})-\infty1_{\left\{  r^{a}/T\right\}  ^{c}}(R^{a}/T)\}\right]  .
\]
Since $I^{\infty}$ is convex, to prove the lower bound for (\ref{eqn:Laplace})
we can assume $F$ is convex and lower semicontinuous \cite[Theorem
1.2.1]{dupell4}. The representation from Lemma \ref{lem:rep} will be applied.
We first note that the representation will include a term of the form
$\infty1_{\left\{  r^{a}/T\right\}  ^{c}}(\bar{R}^{a}/T)$ on the right hand
side. We can remove this term if we restrict the infimum to controls for which
$\bar{R}^{a}=r^{a}$ w.p.1, and do so for notational convenience. The
representation thus becomes
\begin{align}
-  &  \frac{1}{T}\log E\left[  \exp\{-TF(\eta_{T}^{a})-\infty1_{\left\{
r^{a}\right\}  ^{c}}(R^{a})\}\right] \label{eqn:fixed_r_RE}\\
=  &  \inf E\rule{0pt}{24pt}\left[  \rule{0pt}{24pt}\rule{0pt}{24pt}%
\rule{0pt}{24pt}F(\bar{\eta}_{T}^{a})+\frac{1}{T}\sum_{i=0}^{r^{a}-1}\left[
R\left(  \bar{\alpha}_{i}(\boldsymbol{\bar{U}}^{a}(i),\cdot|\bar{M}_{0}%
^{a}(i))\left\Vert \alpha(\boldsymbol{\bar{U}}^{a}(i),\cdot|\bar{M}_{0}%
^{a}(i))\right.  \right)  \right.  \right. \nonumber\\
\quad &  \left.  +R\left(  \bar{\beta}_{i}^{a}\left\Vert \beta^{a}\right.
\right)  \right]  +\frac{1}{T}\sum_{i=0}^{r^{a}-1}\sum_{k=0}^{\bar{N}_{i}%
^{a}-1}\left[  R\left(  \bar{p}_{i,k}(\bar{M}_{k}^{a}(i),\boldsymbol{\cdot
}|\boldsymbol{\bar{U}}^{a}(i))\left\Vert p(\bar{M}_{k}^{a}%
(i),\boldsymbol{\cdot}|\boldsymbol{\bar{U}}^{a}(i))\right.  \right)  \right.
\nonumber\\
&  \left.  +R\left(  \bar{\sigma}_{i,k}^{a}\left\Vert \sigma^{a}\right.
\right)  \rule{0pt}{24pt}\rule{0pt}{24pt}\right]  \rule{0pt}{24pt}%
\rule{0pt}{24pt}\nonumber
\end{align}
There are four relative entropy sums in (\ref{eqn:fixed_r_RE}). Since $F$ is
bounded from below, the lower bound holds vacuously unless each such term is
uniformly bounded.

We first show that the empirical distribution on $\bar{N}_{i}^{a}$ converges
to $\delta_{\infty}$ by using a martingale argument. For $C_{1},C_{2}%
\subset\lbrack0,\infty)$, let
\[
\xi^{T}(C_{1}\times C_{2})\doteq\frac{1}{r^{a}}\sum_{i=0}^{r^{a}-1}%
\delta_{\bar{N}_{i}^{a}}(C_{1})\bar{\beta}_{i}^{a}(C_{2}).
\]
Consider the one-point compactification of $[0,\infty)$, which we identify
with $[0,\infty]$, and left $f:[0,\infty]\rightarrow\mathbb{R}$ be bounded and
continuous. Then for any $\varepsilon>0$
\begin{align}
P  &  \left\{  \left\vert \int[f(z_{1})-f(z_{2})]\xi^{T}(dz_{1}\times
dz_{2})\right\vert \geq\varepsilon\right\} \label{eqn:mg}\\
&  =P\left\{  \left\vert \frac{1}{r^{a}}\sum_{i=0}^{r^{a}-1}\left(  f(\bar
{N}_{i}^{a})-\int f(n)\bar{\beta}_{i}^{a}(dn)\right)  \right\vert
\geq\varepsilon\right\} \nonumber\\
&  \leq\frac{1}{\varepsilon^{2}}E\left[  \left\vert \frac{1}{r^{a}}\sum
_{i=0}^{r^{a}-1}\left(  f(\bar{N}_{i}^{a})-\int f(n)\bar{\beta}_{i}%
^{a}(dn)\right)  \right\vert ^{2}\right] \nonumber\\
&  \leq\frac{\left\Vert f\right\Vert _{\infty}^{2}}{\varepsilon^{2}r^{a}%
}\nonumber\\
&  \rightarrow0,\nonumber
\end{align}
and thus any weak limit of $\xi^{T}$ has identical marginals. Next note that
by Jensen's inequality and since $r^{a}/T\rightarrow A\in(0,\infty)$, the
uniform bound on the second relative entropy sum implies that $ER\left(
[\xi^{T}]_{2}\left\Vert \beta^{a}\right.  \right)  $ is uniformly bounded.
Using that inf$\{R\left(  \gamma\left\Vert \beta_{a}\right.  \right)  :\int
u\gamma\left(  du\right)  =b\}=b\log\frac{b}{a}+\left(  1+b\right)  \log
\frac{1+a}{1+b}$, is follows that the weak limit of $[\xi^{T}]_{2}%
=\delta_{\infty}$ w.p.1.

Next we consider the asymptotic properties of the collection $\left\{  \bar
{M}_{\ell}^{a}(i)\right\}  $, under boundedness of the third relative entropy
sum. We use that the empirical measure of the $\left\{  \bar{N}_{i}%
^{a}\right\}  $ tends to $\delta_{\infty}$. Since $p(\cdot,\cdot
|\boldsymbol{\bar{U}}^{a}(i))$ is the transition function of a finite state
Markov chain this means that asymptotically the $\bar{M}_{k}^{a}%
(i),k=0,\ldots,\bar{N}_{i}^{a}$ are samples from the transition probability
(\ref{eqn:trans_probs}) with $\boldsymbol{x=\bar{U}}^{a}(i)$, and in
particular that $\bar{M}_{0}^{a}(i+1)$ is asymptotically conditionally
independent with distribution $\left(  \rho(\boldsymbol{x}),1-\rho
(\boldsymbol{x})\right)  $. Indeed, a martingale argument similar to
(\ref{eqn:mg}) shows that if%
\[
\mu^{T}(C_{1}\times C_{2})\doteq\frac{1}{r^{a}}\sum_{i=0}^{r^{a}-1}%
\delta_{\boldsymbol{\bar{U}}^{a}(i)}(C_{1})\delta_{\bar{M}_{0}^{a}(i)}%
(C_{2}),
\]
and if $\mu^{\infty}$ is the weak limit of any convergent subsequence (which
must exist by compactness), then
\begin{equation}
\left(  \lbrack\mu^{\infty}]_{2|1}(\left\{  0\right\}  |\boldsymbol{y}%
),[\mu^{\infty}]_{2|1}(\left\{  1\right\}  |\boldsymbol{y})\right)  =\left(
\rho(\boldsymbol{y}),1-\rho(\boldsymbol{y})\right)  \text{ }[\mu^{\infty}%
]_{1}\text{-a.s.,} \label{eqn:rho_marg}%
\end{equation}
and the same is true for the empirical measure of $\left\{  \bar{M}_{k}%
^{a}(i),k=0,\ldots,\bar{N}_{i}^{a}\right\}  $.

We next remove the third relative entropy sum from the representation
(\ref{eqn:fixed_r_RE}), and obtain a lower bound for the right hand side using
Lemma \ref{lem:RE_decomp}. Let $\hat{\sigma}_{i}^{a}$ denote the distribution
of the random variable
\[
\hat{\tau}_{i}^{a}=\sum_{\ell=0}^{\bar{N}_{i}^{a}-1}\bar{\tau}_{i,\ell}^{a}%
\]
Then we have the lower bound%
\begin{align}
-  &  \frac{1}{T}\log E\left[  \exp\{-TF(\eta_{T}^{a})-\infty1_{\left\{
r^{a}\right\}  ^{c}}(R^{a})\}\right] \label{lb}\\
&  \geq\inf\rule{0pt}{24pt}\left[  \rule{0pt}{24pt}\rule{0pt}{24pt}%
F(E\bar{\eta}_{T}^{a})+\frac{1}{T}E\sum_{i=0}^{r^{a}-1}\left[  R\left(
\bar{\alpha}_{i}(\boldsymbol{\bar{U}}^{a}(i),\cdot|\bar{M}_{0}^{a}%
(i))\left\Vert \alpha(\boldsymbol{\bar{U}}^{a}(i),\cdot|\bar{M}_{0}%
^{a}(i))\right.  \right)  \right.  \right. \nonumber\\
&  \left.  \left.  +R\left(  \hat{\sigma}_{i}^{a}\left\Vert \sigma^{1}\right.
\right)  \right]  \rule{0pt}{24pt}\rule{0pt}{24pt}\right]  \rule{0pt}{24pt}%
\rule{0pt}{24pt}.\nonumber
\end{align}
To study the lower bound of (\ref{lb}) we introduce the measures%
\begin{align*}
\kappa^{T}(C_{1}\times C_{2}\times C_{3})  &  \doteq\frac{1}{r^{a}}\sum
_{i=0}^{r^{a}-1}\delta_{\boldsymbol{\bar{U}}^{a}(i)}(C_{1})\bar{\alpha}%
_{i}(\boldsymbol{\bar{U}}^{a}(i),C_{2}|\bar{M}_{0}^{a}(i))\delta_{\bar{M}%
_{0}^{a}(i)}(C_{3})\\
\xi^{T}(C_{1}\times C_{2})  &  \doteq\frac{1}{T}\sum_{i=0}^{r^{a}-1}%
\delta_{\boldsymbol{\bar{U}}^{a}(i)}(C_{1})\delta_{(\hat{m}_{i}^{a})^{-1}%
}(C_{2})\hat{m}_{i}^{a},
\end{align*}
where $\hat{m}_{i}^{a}$ is the conditional mean of $\hat{\tau}_{i}^{a}$. The
restriction on the control measures implies
\[
\sum_{i=0}^{r^{a}-2}\hat{\tau}_{i}^{a}\leq T<\sum_{i=0}^{r^{a}-1}\hat{\tau
}_{i}^{a},
\]
and since function $h$ is increasing on $\left(  1,\infty\right)  $, we can
assume without loss of generality that $\hat{m}_{r^{a}-1}^{a}\leq1$.

We introduce the function $\ell:%
%TCIMACRO{\U{211d} }%
%BeginExpansion
\mathbb{R}
%EndExpansion
_{+}\rightarrow%
%TCIMACRO{\U{211d} }%
%BeginExpansion
\mathbb{R}
%EndExpansion
_{+}$ given by $\ell\left(  b\right)  =b\log b-b+1$. Note that $h\left(
b\right)  =b\ell\left(  1/b\right)  $. The sequence $\left\{  \kappa
^{T}\right\}  $ is tight because $S$ and $\{0,1\}$ are compact. Using the fact
that $\hat{\sigma}_{i}^{a}$ selects the conditional distribution of $\hat
{\tau}_{i}^{a}$ and $\inf\{ R\left(  \gamma\left\Vert \sigma^{1}\right.
\right)  :\int u\gamma\left(  du\right)  $ $=b\} =h\left(  b\right)  $, we
have
\begin{align*}
E\left[  \frac{1}{T}\sum_{i=0}^{r^{a}-1}R\left(  \hat{\sigma}_{i}%
^{a}\left\Vert \sigma^{1}\right.  \right)  \right]   &  \geq E\left[  \frac
{1}{T}\sum_{i=0}^{r^{a}-1}h(\hat{m}_{i}^{a})\right] \\
&  =E\left[  \frac{1}{T}\sum_{i=0}^{r^{a}-1}\hat{m}_{i}^{a}\ell([\hat{m}%
_{i}^{a}]^{-1})\right] \\
&  =E\left[  \int\ell(z)\xi^{T}(du\times dz)\right]  .
\end{align*}
Hence the uniform bound on the relative entropy sum gives that $\left\{
\xi^{T}\right\}  $ is tight. Let $f:S\rightarrow\mathbb{R}$ be bounded and
continuous. Since $\bar{\alpha}_{i}(\boldsymbol{\bar{U}}^{a}(i),\cdot|\bar
{M}_{0}^{a}(i))$ selects the conditional distribution of $\boldsymbol{\bar{U}%
}^{a}(i+1)$, for $\varepsilon>0$
\begin{align*}
P  &  \left\{  \left\vert \sum_{z=1}^{2}\int[f(\boldsymbol{y}_{1}%
)-f(\boldsymbol{y}_{2})]\kappa^{T}(d\boldsymbol{y}_{1}\times d\boldsymbol{y}%
_{2}\times dz)\right\vert \geq\varepsilon\right\} \\
&  =P\left\{  \left\vert \frac{1}{r^{a}}\sum_{i=0}^{r^{a}-1}\left(
f(\boldsymbol{\bar{U}}^{a}(i+1))-\int f(\boldsymbol{y})\bar{\alpha}%
_{i}(\boldsymbol{\bar{U}}^{a}(i),d\boldsymbol{y}|\bar{M}_{0}^{a}(i))\right)
\right\vert \geq\varepsilon\right\} \\
&  \leq\frac{1}{\varepsilon^{2}}E\left[  \left\vert \frac{1}{r^{a}}\sum
_{i=0}^{r^{a}-1}\left(  f(\boldsymbol{\bar{U}}^{a}(i+1))-\int f(\boldsymbol{y}%
)\bar{\alpha}_{i}(\boldsymbol{\bar{U}}^{a}(i),d\boldsymbol{y}|\bar{M}_{0}%
^{a}(i))\right)  \right\vert ^{2}\right] \\
&  \leq\frac{\left\Vert f\right\Vert _{\infty}^{2}}{\varepsilon^{2}r^{a}}\\
&  \rightarrow0.
\end{align*}
We find that
\begin{equation}
\lbrack\kappa^{\infty}]_{1}=[\kappa^{\infty}]_{2}\text{ w.p.1}.
\label{eqn:invariance_prop}%
\end{equation}
Using Fatou's lemma and the definition of $\varphi$, we get the lower bound
$AR([\kappa^{\infty}]_{1,2}\Vert$ $\lbrack\kappa^{\infty}]_{1}\otimes\varphi)$
on the weak limit of the corresponding relative entropies (the second sum in
(\ref{lb})).

With regard to $\xi^{\infty}$, comparing the form of $\left[  \xi^{T}\right]
_{1}$ and $\bar{\psi}_{T}^{a}$ gives
\[
E\left[  \xi^{\infty}\right]  _{1}=E\bar{\psi}_{\infty}.
\]
Observe that
\[
\int z\xi^{T}(d\boldsymbol{x}\times dz)=\frac{r^{a}}{T}\left[  \kappa
^{T}\right]  _{1}(d\boldsymbol{x}).
\]
Because of the superlinearity of $\ell$ we have uniform integrability, and
thus passing to the limit gives
\[
\int z[\xi^{\infty}]_{1}(d\boldsymbol{x})[\xi^{\infty}]_{2|1}%
(dz|\boldsymbol{x})=A[\kappa^{\infty}]_{1}(d\boldsymbol{x}).
\]
Hence with the definition $b(\boldsymbol{x})=\int z[\xi^{\infty}%
]_{2|1}(dz|\boldsymbol{x})$, $[d[\kappa^{\infty}]_{1}/d[\xi^{\infty}%
]_{1}](\boldsymbol{x})=b(\boldsymbol{x})/A$. Using Fatou's lemma, Jensen's
inequality and the weak convergence, we get the following lower bound on the
corresponding relative entropies (the first sum in (\ref{lb})):%
\begin{align*}
\liminf_{T\rightarrow\infty}  &  \int\ell(z)\xi^{T}(du\times dz)\\
&  \geq\int\ell(z)\xi^{\infty}(du\times dz)\\
&  \geq\int\ell\left(  \int z[\xi^{\infty}]_{2|1}(dz|\boldsymbol{x})\right)
[\xi^{\infty}]_{1}(d\boldsymbol{x})\\
&  =\int\ell\left(  b(\boldsymbol{x})\right)  [\xi^{\infty}]_{1}%
(d\boldsymbol{x})\\
&  =A\int h(1/b(\boldsymbol{x}))[\kappa^{\infty}]_{1}(d\boldsymbol{x}).
\end{align*}
Let $r(\boldsymbol{y})=A/b\left(  \boldsymbol{y}\right)  $. Then we can write
the combined lower bound on the relative entropies as
\begin{align}
&  AER([\kappa^{\infty}]_{1,2}\Vert\lbrack\kappa^{\infty}]_{1}\otimes
\varphi)+AE\int h(1/b(\boldsymbol{x}))[\kappa^{\infty}]_{1}(d\boldsymbol{x}%
)\nonumber\\
&  =AER([\kappa^{\infty}]_{1,2}\Vert\lbrack\kappa^{\infty}]_{1}\otimes
\varphi)-AE\int\log r(\boldsymbol{y})[\kappa^{\infty}]_{1}(d\boldsymbol{y}%
)+A\log A-A+1. \label{eqn:final_lb}%
\end{align}

We next consider $\bar{\psi}_{T}^{a}$ and $\bar{\eta}_{T}^{a}$. Using the
limiting properties of the $\bar{M}_{k}^{a}(i)$, we have
\begin{align*}
\bar{\eta}_{T}^{a}(C)  &  =\frac{1}{T}\int_{0}^{T}\left[  1_{\left\{  \bar
{Z}^{a}(t)=0\right\}  }\delta_{\boldsymbol{\bar{Y}}^{a}(t)}(C)+1_{\left\{
\bar{Z}^{a}(t)=1\right\}  }\delta_{\boldsymbol{\bar{Y}}^{a}(t)^{R}}(C)\right]
dt\\
&  \rightarrow\int\left[  \rho(\boldsymbol{y})\delta_{\boldsymbol{y}%
}(C)+(1-\rho(\boldsymbol{y}))\delta_{\boldsymbol{y}^{R}}(C)\right]  \left[
\xi^{\infty}\right]  _{1}\\
&  =\bar{\eta}_{\infty}(C)
\end{align*}
and
\[
\bar{\psi}_{T}^{a}(C)=\frac{1}{T}\int_{0}^{T}\delta_{\boldsymbol{\bar{Y}}%
^{a}(t)}(C)dt\rightarrow\int_{C}\left[  \xi^{\infty}\right]  _{1}=\bar{\psi
}_{\infty}(C).
\]
Note that this implies the relation%
\begin{equation}
\bar{\eta}_{\infty}(C)=\int_{C}\left[  \rho(\boldsymbol{y})\bar{\psi}_{\infty
}(d\boldsymbol{y})+\rho(\boldsymbol{y}^{R})\bar{\psi}_{\infty}(d\boldsymbol{y}%
^{R})\right]  . \label{eqn:eta_psi_relation}%
\end{equation}

Finally we consider the weighted empirical measure. By (\ref{eqn:final_lb}),
Lemma \ref{lem:rate_relation} and Jensen's inequality we have the lower bound
$\left[  F(E\bar{\eta}_{\infty})+K(E\bar{\psi}_{\infty})\right]  $ for the
limit inferior of the right hand side of (\ref{eqn:fixed_r_RE}), where
$\bar{\eta}_{\infty}$ and $\bar{\psi}_{\infty}$ are related by
(\ref{eqn:eta_psi_relation}). Thus we need only show that
\[
K(E\bar{\psi}_{\infty})\geq I^{0}(E\bar{\eta}_{\infty})=I^{\infty}(E\bar{\eta
}_{\infty}).
\]
The equality follows from the definition of $I^{\infty}$ and
(\ref{eqn:eta_psi_relation}). Let $a(\boldsymbol{y})=[dE\bar{\psi}_{\infty
}/d\bar{\mu}](\boldsymbol{y})$. Then%
\[
K(E\bar{\psi}_{\infty})=1-\int\sqrt{a(\boldsymbol{x})a(\boldsymbol{y})}%
\bar{\mu}(d\boldsymbol{x})\varphi(\boldsymbol{x},d\boldsymbol{y})
\]
and
\[
I^{0}(E\bar{\eta}_{\infty})=1-\int\frac{1}{2}\sqrt{\left(  a(\boldsymbol{x}%
)+a(\boldsymbol{x}^{R})\right)  \left(  a(\boldsymbol{y})+a(\boldsymbol{y}%
^{R})\right)  }\mu(d\boldsymbol{x})\alpha(\boldsymbol{x},d\boldsymbol{y}).
\]
Using that $\rho(\boldsymbol{x})=1-\rho(\boldsymbol{x}^{R})$ and symmetry
\begin{align*}
&  \int\sqrt{a(\boldsymbol{x})a(\boldsymbol{y})}\frac{1}{2}\left[
\mu(d\boldsymbol{x})+\mu(d\boldsymbol{x}^{R})\right]  \left(  \rho
(\boldsymbol{x})\alpha(\boldsymbol{x},d\boldsymbol{y})+\rho(\boldsymbol{x}%
^{R})\alpha(\boldsymbol{x}^{R},d\boldsymbol{y}^{R})\right) \\
&  =\frac{1}{2}\int\left(  \sqrt{a(\boldsymbol{x})a(\boldsymbol{y})}%
+\sqrt{a(\boldsymbol{x}^{R})a(\boldsymbol{y}^{R})}\right)  \mu(d\boldsymbol{x}%
)\left(  \rho(\boldsymbol{x})\alpha(\boldsymbol{x},d\boldsymbol{y}%
)+\rho(\boldsymbol{x}^{R})\alpha(\boldsymbol{x}^{R},d\boldsymbol{y}%
^{R})\right) \\
&  =\frac{1}{2}\int\left(  \sqrt{a(\boldsymbol{x})a(\boldsymbol{y})}%
+\sqrt{a(\boldsymbol{x}^{R})a(\boldsymbol{y}^{R})}\right)  \mu(d\boldsymbol{x}%
)\alpha(\boldsymbol{x},d\boldsymbol{y}).
\end{align*}
We now use
\[
\sqrt{a(\boldsymbol{x})a(\boldsymbol{y})}+\sqrt{a(\boldsymbol{x}%
^{R})a(\boldsymbol{y}^{R})}\leq\sqrt{\left(  a(\boldsymbol{x}%
)+a(\boldsymbol{x}^{R})\right)  \left(  a(\boldsymbol{y})+a(\boldsymbol{y}%
^{R})\right)  }
\]
to obtain $K(E\bar{\psi}_{\infty})\geq I^{0}(E\bar{\eta}_{\infty})$. [We note
for later use that given $\bar{\eta}$ that is absolutely continuous with
respect to Lebesgue measure and such that $I^{\infty}(\bar{\eta})<\infty$,
$K(\bar{\psi})=I^{0}(\bar{\eta})$ can be shown for a $\bar{\psi}$ that maps to
$\bar{\eta}$ by taking $\bar{\psi}=[\bar{\eta}+\bar{\eta}^{R}]/2$.]

\subsubsection{Combining the cases}

In the last subsection we showed that for any sequence $r^{a}$ such that
$r^{a}/T\rightarrow A\in\lbrack1/C,C]$,
\begin{align*}
\liminf_{T\rightarrow\infty}-\frac{1}{T}\log E\left[  \exp\{-TF(\eta_{T}%
^{a})-\infty1_{\left\{  r^{a}/T\right\}  ^{c}}(R^{a}/T)\}\right]   &
\geq\left[  F(E\bar{\eta}_{\infty})+I^{\infty}(E\bar{\eta}_{\infty})\right] \\
&  \geq\inf_{\nu\in\mathcal{P}(S)}\left[  F(\nu)+I^{\infty}(\nu)\right]  ,
\end{align*}
and an argument by contradiction shows that the bound is uniform in $A$. Thus
\begin{align*}
\liminf_{T\rightarrow\infty}  &  -\frac{1}{T}\log\left\{  \sum_{r^{a}%
=\left\lceil \frac{T}{C}\right\rceil }^{\left\lfloor TC\right\rfloor }E\left[
\exp\{-TF(\eta_{T}^{a})-\infty1_{\left\{  r^{a}/T\right\}  ^{c}}%
(R^{a}/T)\}\right]  \right\} \\
&  \geq\liminf_{T\rightarrow\infty}-\frac{1}{T}\log\left\{  TC\cdot
\bigvee_{r^{a}=\left\lceil \frac{T}{C}\right\rceil }^{\left\lfloor
TC\right\rfloor }E\left[  \exp\{-TF(\eta_{T}^{a})-\infty1_{\left\{
r^{a}/T\right\}  ^{c}}(R^{a}/T)\}\right]  \right\} \\
&  \geq\inf_{\nu\in\mathcal{P}(S)}\left[  F(\nu)+I^{\infty}(\nu)\right]  .
\end{align*}
We now partition $E\left[  \exp\{-TF(\eta_{T}^{a})\right]  $ according to the
various cases to obtain the overall lower bound
\begin{align*}
\liminf_{T\rightarrow\infty}  &  -\frac{1}{T}\log E\left[  \exp\{-TF(\eta
_{T}^{a_{T}})\}\right] \\
&  \geq\min\left\{  \inf_{\nu\in\mathcal{P}(S)}\left[  F(\nu)+I^{\infty}%
(\nu)\right]  ,\frac{h\left(  C\right)  }{C},\left(  C-1\right)  h\left(
\frac{1}{C-1}\right)  \right\}  .
\end{align*}
Letting $C\rightarrow\infty$ and using the fact that $I^{\infty}\leq1$, we
have the desired lower bound%
\[
\liminf_{T\rightarrow\infty}-\frac{1}{T}\log E\left[  \exp\{-TF(\eta
_{T}^{a_{T}})\}\right]  \geq\inf_{\nu\in\mathcal{P}(S)}\left[  F(\nu
)+I^{\infty}(\nu)\right]  .
\]

\subsection{Upper bound}

The proof of the reverse inequality
\begin{equation}
\limsup_{T\rightarrow\infty}-\frac{1}{T}\log E\left[  \exp\{-TF(\eta
_{T}^{a_{T}})\}\right]  \leq\inf_{\nu\in\mathcal{P}(S)}\left[  F(\nu
)+I^{\infty}(\nu)\right]  \label{eqn:Laplace_UB}%
\end{equation}
is simpler. Let bounded and continuous $F$ be given. Given $\varepsilon>0$, we
can find $\nu$ that is absolutely continuous with respect to Lebesgue measure
and for which
\[
\left[  F(\nu)+I^{\infty}(\nu)\right]  \leq\inf_{\nu\in\mathcal{P}(S)}\left[
F(\nu)+I^{\infty}(\nu)\right]  +\varepsilon.
\]
Next we use that $I^{\infty}$ is convex and $I^{\infty}(\mu)=0$ to find
$\tau>0$ such that
\[
\left[  F(\nu^{\tau})+I^{\infty}(\nu^{\tau})\right]  \leq\left[
F(\nu)+I^{\infty}(\nu)\right]  +\varepsilon
\]
when $\nu^{\tau}=\tau\mu+(1-\tau)\nu$. Note that if $\theta(\boldsymbol{x}%
)=[d\nu/d\mu](\boldsymbol{x})$ and $\theta^{\tau}(\boldsymbol{x})=[d\nu^{\tau
}/d\mu](\boldsymbol{x})$, then $\theta^{\tau}(\boldsymbol{x})\geq\tau>0$ and
so $I^{\infty}(\nu^{\tau})<1$.

We will construct a control to use in the representation such $\bar{\eta}%
_{T}^{a_{T}}$ will converge w.p.1 to $\nu^{\tau}$ and $RE^{T}$ will converge
w.p.1 to $I^{\infty}(\nu^{\tau})$. These convergences will follow from the
ergodic theorem, and the bound $\theta^{\tau}(\boldsymbol{x})\geq\tau$ will be
used to argue that the ergodicity of the original process is inherited by the
controlled process.

We now proceed to the construction. Let $\varsigma(d\boldsymbol{x})=[\nu
^{\tau}(d\boldsymbol{x})+\nu^{\tau}(d\boldsymbol{x}^{R})]/2$, so that
$I^{\infty}(\nu^{\tau})=K(\varsigma)$. Let $\kappa(\boldsymbol{x}%
)=[d\varsigma/d\bar{\mu}](\boldsymbol{x})\geq\tau/2.$ Then
\[
K(\varsigma)=1-\int\sqrt{\kappa(\boldsymbol{x})\kappa(\boldsymbol{y})}\bar
{\mu}(d\boldsymbol{x})\varphi(\boldsymbol{x},d\boldsymbol{y}).
\]
Since $\kappa$ is uniformly bounded from below we have the dual relationship%
\begin{align*}
&  -\log\int\sqrt{\kappa(\boldsymbol{x})\kappa(\boldsymbol{y})}\bar{\mu
}(d\boldsymbol{x})\varphi(\boldsymbol{x},d\boldsymbol{y})\\
&  =-\log\int e^{\frac{1}{2}\left[  \log\kappa(\boldsymbol{x})+\log
\kappa(\boldsymbol{y})\right]  }\bar{\mu}(d\boldsymbol{x})\varphi
(\boldsymbol{x},d\boldsymbol{y})\\
&  =R(\eta\Vert\bar{\mu}\otimes\varphi)-\frac{1}{2}\int\left[  \log
\kappa(\boldsymbol{x})+\log\kappa(\boldsymbol{y})\right]  \nu(d\boldsymbol{x}%
,d\boldsymbol{y}),
\end{align*}
where
\[
\eta(C)=\frac{\int_{C}e^{\frac{1}{2}\left[  \log\kappa(\boldsymbol{x}%
)+\log\kappa(\boldsymbol{y})\right]  }\bar{\mu}(d\boldsymbol{x})\varphi
(\boldsymbol{x},d\boldsymbol{y})}{\int_{S^{2}}e^{\frac{1}{2}\left[  \log
\kappa(\boldsymbol{x})+\log\kappa(\boldsymbol{y})\right]  }\bar{\mu
}(d\boldsymbol{x})\varphi(\boldsymbol{x},d\boldsymbol{y})}.
\]
Since $\sqrt{\kappa(\boldsymbol{x})\kappa(\boldsymbol{y})}$ is symmetric in
$\boldsymbol{x}$ and $\boldsymbol{y}$ automatically $[\eta]_{1}=[\eta]_{2}$,
and using the bound $\kappa(\boldsymbol{x})\geq\tau/2$ we can factor
$\eta(d\boldsymbol{x}\times d\boldsymbol{y})=[\eta]_{1}(d\boldsymbol{x}%
)\bar{\varphi}(\boldsymbol{x},d\boldsymbol{y})$, where $\bar{\varphi}$ is the
transition kernel of a uniformly ergodic Markov chain. Let%
\[
\bar{\alpha}(\boldsymbol{x},d\boldsymbol{y}|0)=\bar{\alpha}(\boldsymbol{x}%
,d\boldsymbol{y}|1)=\bar{\varphi}(\boldsymbol{x},d\boldsymbol{y}).
\]
Notice that from the definition of $\bar{\varphi}$, $\bar{\varphi
}(\boldsymbol{x},d\boldsymbol{y})=\bar{\varphi}(\boldsymbol{x}^{R}%
,d\boldsymbol{y}^{R})$, hence $\bar{\alpha}$ has the property that
\[
\bar{\alpha}(\boldsymbol{x}^{R},d\boldsymbol{y}^{R}|0)=\bar{\alpha
}(\boldsymbol{x},d\boldsymbol{y}|1)
\]
These will be the transition kernels used to construct the $\boldsymbol{\bar
{U}}^{a}(i)$.

Now of course the invariant distribution of $\bar{\varphi}$ is $[\eta]_{1}$,
and not the desired distribution $\varsigma$. Let $r(\boldsymbol{x}%
)=[d\varsigma/d[\eta]_{1}](\boldsymbol{x})$. Then $r(\boldsymbol{x})$
identifies the way in which the distribution of the random variables $\bar
{N}_{j}^{a}$ should be modified so that in the continuous time the empirical
measure $[\eta]_{1}$ is reshaped into $\varsigma$. Choose $A$ so that equality
holds in (\ref{eqn:two_rates}), and set $b(\boldsymbol{x})=Ar(\boldsymbol{x}%
)$. Then $\bar{\beta}_{i}^{a}$ is chosen to be $\beta^{ab(\boldsymbol{x})}$.
Consistent with the analysis of the upper bound, we do not perturb the
distribution of the other variables, so that $\bar{p}_{i,k}(\cdot
,\boldsymbol{\cdot}|\boldsymbol{\cdot})=p(\cdot,\boldsymbol{\cdot
}|\boldsymbol{\cdot})$ and $\bar{\sigma}_{j,\ell}^{a}=\sigma^{a}$. We then
construct the controlled processes using these measures in exact analogy with
the construction of the original process.

With this choice and Lemma \ref{lem:rep} we obtain the bound
\begin{align*}
&  -\frac{1}{T}\log E\left[  \exp\{-TF(\eta_{T}^{a_{T}})\}\right] \\
&  \leq E\rule{0pt}{24pt}\left[  \rule{0pt}{24pt}F(\bar{\eta}_{T}^{a}%
)+\frac{1}{T}\sum_{i=0}^{\bar{R}^{a}-1}\left[  R\left(  \bar{\alpha
}(\boldsymbol{\bar{U}}^{a}(i),\cdot|\bar{M}_{0}^{a}(i))\left\Vert
\alpha(\boldsymbol{\bar{U}}^{a}(i),\cdot|\bar{M}_{0}^{a}(i))\right.  \right)
\right.  \right. \\
&  \left.  \left.  +R\left(  \beta^{ab(\boldsymbol{\bar{U}}^{a}(i))}\left\Vert
\beta^{a}\right.  \right)  \right]  \rule{0pt}{24pt}\right]  \rule{0pt}{24pt}.
\end{align*}
Now apply the ergodic theorem, and use that w.p.1 $A^{T}=\bar{R}%
^{a}/T\rightarrow1/\int b(\boldsymbol{x})[\eta]_{1}(d\boldsymbol{x})$ $=A$,
and also that asymptotically the conditional distribution of $\bar{M}_{0}%
^{a}(i)$ is given by $(\rho_{0}(\boldsymbol{\bar{U}}^{a}(i)),\rho
_{1}(\boldsymbol{\bar{U}}^{a}(i)))$. Then right hand side of the last display
converges to we have the limit%
\begin{align*}
&  F(\nu^{\tau})+\frac{1}{\int b(\boldsymbol{x})[\eta]_{1}(d\boldsymbol{x}%
)}\int\left[  \sum_{z=1}^{2}R\left(  \bar{\alpha}(\boldsymbol{x}%
,d\boldsymbol{y}|z)\left\Vert \alpha(\boldsymbol{x},d\boldsymbol{y}|z)\right.
\right)  \rho_{z}(\boldsymbol{x})\right]  r(\boldsymbol{x})[\eta
]_{1}(d\boldsymbol{x})\\
&  \quad+\frac{1}{\int b(\boldsymbol{x})[\eta]_{1}(d\boldsymbol{x})}%
\int\left(  -\log b(\boldsymbol{x})+b(\boldsymbol{x})-1\right)
b(\boldsymbol{x})[\eta]_{1}(d\boldsymbol{x})\\
&  =F(\nu^{\tau})+A\int R\left(  \bar{\varphi}(\boldsymbol{x},d\boldsymbol{y}%
)\left\Vert \varphi(\boldsymbol{x},d\boldsymbol{y})\right.  \right)
\varsigma(d\boldsymbol{x})-A\int\log r(\boldsymbol{x})\varsigma
(d\boldsymbol{x})+A\log A-A+1\\
&  =F(\nu^{\tau})+K(\varsigma)\\
&  =F(\nu^{\tau})+I^{\infty}(\nu^{\tau}).
\end{align*}
This completes the proof of (\ref{eqn:Laplace_UB}) and also the proof of
Theorem \ref{thm:LDP_MP}.

\end{document}